%% file: main.tex
\documentclass{amsart}
\usepackage{graphicx} 

\usepackage{amsmath,amssymb,amsthm,bm,bbm}
\usepackage{amscd}
\usepackage{color}

\usepackage{comment}

\usepackage[left=2cm,right=2cm]{geometry}
\usepackage{enumitem}
\setlist[itemize,enumerate]{labelindent=1em, leftmargin=*}

\usepackage{tikz}
\usepackage{tikz-cd}

\renewcommand{\epsilon}{\varepsilon}
\renewcommand{\phi}{\varphi}

\bmdefine{\ba}{a}
\bmdefine{\bb}{b}
\bmdefine{\be}{e}
\bmdefine{\bE}{E}
\bmdefine{\bh}{h}
\bmdefine{\bg}{g}
\bmdefine{\bk}{k}
\bmdefine{\bp}{p}
\bmdefine{\bx}{x}
\bmdefine{\by}{y}
\bmdefine{\bs}{s}
\bmdefine{\bt}{t}
\bmdefine{\bu}{u}
\bmdefine{\bv}{v}
\bmdefine{\bw}{w}
\bmdefine{\bep}{\epsilon}

\newcommand{\R}{{\mathbb R}}
\newcommand{\RR}{{\mathcal R}}
\newcommand{\Q}{{\mathbb Q}}
\newcommand{\Z}{{\mathbb Z}}
\newcommand{\C}{{\mathbb C}}
\newcommand{\CC}{{\mathcal C}}

\newcommand{\N}{{\mathbb N}}
\newcommand{\NN}{{\mathcal N}}

\newcommand{\A}{{\mathcal A}}
\newcommand{\F}{{\mathcal F}}
\newcommand{\I}{{\mathcal I}}

\newcommand{\W}{{\mathcal W}}

\renewcommand{\Re}{{\mathrm{Re}}}

\newcommand{\NS}[1]{\NN\left( #1 \right)}
\newcommand{\Na}{\NS{\alpha}}
\newcommand{\Nb}{\NS{\beta}}
\newcommand{\Pisot}{{\mathcal{S}}}

\newcommand{\ov}{\overline}

\newcommand{\wi}{\widetilde}

\newcommand{\bP}{\boldsymbol{\Pi}}
\newcommand{\bS}{\boldsymbol{\Sigma}}
\newcommand{\bD}{\boldsymbol{\Delta}}
\newcommand{\base}{r}

\newcommand{\sB}{\mathcal{B}}

\newtheorem{defi}{DEFINITION}
\newtheorem{theorem}{THEOREM}
\newtheorem{lem}{LEMMA}
\newtheorem{fac}{FACT}

\newtheorem{prop}{PROPOSITION}

\newtheorem{prob}{PROBLEM}
\newtheorem{rem}{REMARK}
\newtheorem{ex}{Example}

\newcommand{\Romannum}[1]{\uppercase\expandafter{\romannumeral#1}}
\makeatletter
  
  \@addtoreset{equation}{section}
\makeatother

\makeatletter
\@namedef{subjclassname@2020}{\textup{2020} Mathematics Subject Classification}
\makeatother

\title{Borel Complexity of the set of vectors normal for a fixed recurrence sequence}

\begin{document}

\begin{abstract}
In this paper, we consider recurrence sequences $x_n=\xi_1 \alpha_1^n+\xi_2 \alpha_2^n$ ($n=0,1,\ldots$) with companion polynomial $P(X)$. For example, the sequence $x_n=\xi_1(4+\sqrt{2})^n+\xi_2(4-\sqrt{2})^n$ satisfies the recurrence $x_{n+2}-8x_{n+1}+14x_n=0$ and has companion polynomial $P(X)=X^2-8X+14=(X-4-\sqrt{2})(X-4+\sqrt{2})$. We call $(\xi_1,\xi_2)$ {\it normal with respect to the recurrence relation determined by $P(X)$} when $(x_n)_{n\ge 0}$ is uniformly distributed modulo one.

 Determining the Borel complexity of the set of normal vectors for a fixed recurrence sequence is unresolved even for most geometric progressions. 
Under certain assumptions, we prove that the set of normal vectors is $\bP_3^0$-complete. 
 A special case is the new result that the sets of numbers normal in base $\alpha$, i.e. $\{\xi\in \R\mid (\xi\alpha^n)_{n\geq 0}\mbox{ is u.d. modulo one.} \}$, are $\bP_3^0$-complete for every real number $\alpha$ with $|\alpha|$ Pisot.
We analyze the fractional parts of recurrence sequences in terms of finite words via certain numeration systems. 
 One of the difficulties in proving the main result is that even when recurrence sequences are uniformly distributed modulo one, it is not known what the average frequencies of the digits in the corresponding digital expansions are or if they even must exist.
\end{abstract}

\thanks{The first author was supported by the JSPS KAKENHI Grant Number 24K06641. The second author was supported by grant 2019/34/E/ST1/00082 for the project ``Set theoretic methods in dynamics and number theory,'' NCN (The National Science Centre of Poland).}

\author[H. Kaneko]{Hajime Kaneko}
\address[H. Kaneko]{Institute of Mathematics, University of Tsukuba, 1-1-1 Tenodai, Tsukuba, Ibaraki, 305-8571, Japan}
\email{kanekoha@math.tsukuba.ac.jp}

\author[B. Mance]{Bill Mance}
\address[B. Mance]{Uniwersytet im. Adama Mickiewicza w Poznaniu,
  Collegium Mathematicum, ul. Umultowska 87, 61-614 Pozna\'{n}, Poland}
\email{william.mance@amu.edu.pl}

\maketitle

\section{Introduction}

Roughly speaking, a \emph{numeration system} assigns to each real number an \emph{expansion}. Here, an expansion is an infinite 
sequence of \emph{digits} coming from some at most countable set.  A real number is \emph{normal} in a numeration system if all 
\emph{asymptotic frequencies} of finite blocks of consecutive digits appearing in the expansion are \emph{typical} for the 
numeration system. To put some more content into this vague description recall that a real number $\xi$ is normal in base $2$ 
if in its binary expansion every block of digits of length $k$ appears with asymptotic frequency $1/2^k$. It follows that for 
every integer $\base\ge 2$ the set of normal numbers in base $\base$ is a first category set of full Lebesgue measure. Also, 
the normal numbers form a Borel set.  The paper~\cite{DSDB20} showed that sets of normal numbers for many commonly used numeration systems are all $\bP_3^0$-complete, but there are some important cases that don't follow from the main theorem of this paper. 

The initial motivation for our paper was to show that for every real number $\alpha$ with $|\alpha| > 1$ the set $\{\xi\in \R\mid (\xi\alpha^n)_{n\geq 0}\mbox{ is u.d. modulo one.} \}$ is $\bP_3^0$-complete . Unfortunately, this problem seems to be very difficult in general, but we were able to prove this for $\alpha$ where $|\alpha|$ is Pisot. Furthermore, we were able to investigate the completeness in the Borel hierarchy of the set of initial values of recurrence sequences which are uniformly distributed modulo one. For instance, we consider the case where $\alpha_1>1$ is an algebraic integer and suppose that $\alpha_2>1$ is the unique conjugate of $\alpha_1$ with absolute value not less than 1. 
Let $P(X)$ be the minimal polynomial of $\alpha_1$. Then we can show that the set $\{(\xi,\eta)\in \R^2\mid \xi\alpha_1^n+\eta\alpha_2^n\mbox{ is u.d. modulo one.} \}$,  called the set of {\it normal vectors with respect to recurrence relation defined by the companion polynomial $P(X)$,} is $\bP_3^0$-complete. For the definition of companion polynomials and normal vectors for  recurrence sequences, see Section \ref{section2}.

Knowing that the sets of normal numbers are Borel it is natural to gauge their complexity using the descriptive hierarchy of Borel sets.
In that hierarchy, the simplest Borel sets are
open ones and their complements (closed sets). On the next level,
there are countable intersections and countable unions of sets at the first level.
These are $G_{\delta}$ and $F_\sigma$ sets, and the third level is
formed by taking countable intersections and unions of sets at the second level. The procedure continues and provides a
stratification of the family of Borel sets into levels corresponding to countable ordinals. It is known that for
an uncountable Polish space
these levels do not collapse: at each level there appear new sets which do
not occur at any lower level of the hierarchy. Thus to every Borel set we can associate its complexity, that is, the lowest
level of the hierarchy at which the set is visible. On the other hand,
determining the position of ``naturally arising'' or ``non-ad hoc''
sets in the hierarchy is a challenging problem. Only a small number of concrete examples are known to
appear only above the third level.

A.~Kechris asked in the
90's whether the set of real numbers that are normal in base two is an
example of a Borel set properly located at the third level, which was
later confirmed by H.~Ki and T.~Linton in \cite{KiLinton}. More precisely, Ki and Linton showed that the set of numbers
that are normal in an integer base $\base\ge 2$ is a $\bP^0_3$-complete set, which means that this set is a countable intersection of $F_\sigma$ sets
and cannot be represented as a countable union of $G_\delta$-sets. Since then many authors have studied
the Borel complexity of various sets related to normal numbers, and
have extended this result in various directions \cite{AireyJacksonManceComplexityNormalPreserves,BecherHeiberSlamanAbsNormal,BecherSlamanNormal}.

A large number of cases were treated in \cite{DSDB20}. 
For each of the following classes of numeration systems, the respective set of normal numbers are $\bP^0_3$-complete
\begin{itemize}
    \item The regular continued fraction expansion
    \item The $\beta$ expansions for all real $\beta>1$
    \item The generalized GLS expansions. These include the $b$-ary expansions and the L\"{u}roth series expansion as special cases.
\end{itemize}
Recently, the set of complex numbers whose Hurwitz continued fractions are normal with respect to the complex Gauss measure was proved to be $\bP_3^0(\C)$-complete \cite{BorelHurwitz}.
Some types of normal numbers for other numeration system can not be considered in the above generality. The Cantor series expansions, often arising from a non-autonomous dynamical system, must be treated separately. Airey, Jackson, and Mance were able to determine the Borel complexity of various sets of normal numbers with respect to the Cantor series expansions in \cite{AireyJacksonManceComplexityCantorSeries}.
Additionally, while the notion of Poisson genericity pertains to $b$-ary expansions, the problem of determining the complexity of this set may not be determined by the main theorem of \cite{DSDB20}. Becher, Jackson, Kwietniak, and Mance were able determine the complexity of this set in \cite{DescPoisson}.

On the other hand, normal numbers in base $b$ are closely related to the distribution of the fractional parts of geometric sequences with common ratio of $b$. 
The fractional parts of geometric sequences with more general real numbers as their common ratio have also been studied by many mathematicians.
However, it is generally difficult to investigate the fractional parts of a geometric sequence by using numeration systems.
The main purpose of this paper is to investigate the Borel complexity of sets related to the uniformity of fractional parts of geometric sequences and more general recurrence sequences.

We now introduce the notation used throughout this paper.
For every real number $\xi$, there exists a unique pair of 
an integer $u(\xi)$ and a real number $e(\xi)\in [-1/2,1/2)$ such that 
$\xi=u(\xi)+e(\xi)$. 
Putting 
\[
\|\xi\|:=\min\{|\xi-m|\mid m\in \Z\},
\]
we have $\|\xi\|=|e(\xi)|$. For every complex number $z$, 
we denote its complex conjugate by $\ov{z}$. 
Moreover, for every polynomial $h(X)=\sum_{j=0}^{d} c_j X^j$ with 
complex coefficients, we put $\ov{h}(X):=\sum_{j=0}^{d} \ov{c_j} X^j$. 


\subsection{Distribution of geometric sequences modulo one}
Let $\mathbf{x}=(x_n)_{n\geq 0}$ be a sequence of real numbers. 
Recall that $\mathbf{x}$ is {\it uniformly distributed modulo one} if for every pair of real numbers $a,b$ with $-1/2\leq a<b\leq 1/2$, we have 
\[
\lim_{N\to\infty}\frac1{N}\mbox{Card}\{
n<N\mid e(x_n)\in [a,b)
\}=b-a.
\]
Weyl \cite{Weyl4} showed for every real number $\alpha>1$ that $(\xi\alpha^n)_{n\geq 0}$ is uniformly distributed modulo one for almost all real numbers $\xi$, which  also can be deduced from Koksma's metrical criterion of uniform distribution modulo one \cite{Ko35} (See also Lemma \ref{lem:koksma}). A similar result holds for every real number $\alpha$ with $\alpha<-1$. 
Set 
\[
\Na:=\{\xi\in \R\mid (\xi\alpha^n)_{n\geq 0}\mbox{ is u.d. modulo one.} \}
\]
Members of $\Na$ are called {\it normal in base $\alpha$.      }If $\alpha$ is an integer $b$, then $\NS{b}$ is the set of normal numbers in base $b$, which was proved by Wall \cite{Wall}. However, for general real numbers $\alpha$, it is difficult to characterize $\Na$ in terms of dynamical systems on compact spaces. For instance, $\xi\in \Na$ does not imply $\xi+1\in \Na$.
It is difficult to characterize $\Na$ via a suitable invariant measure. 

It is generally difficult to determine if a real number $\xi$ is normal in a fixed base $\alpha$.
Recall that an algebraic integer $\alpha>1$ is called a {\it Pisot number} if the conjugates of $\alpha$, except itself, have absolute values less than 1. Let $\Pisot$ be the set of Pisot numbers. Note that every integer greater than 1 is included in $\Pisot$. It is well known for every $\alpha\in \Pisot$ that $\lim_{n\to\infty} \|\alpha^n\|=0$. In particular, we have $1\not\in \Na$. It is unknown whether there exists $\alpha>1$ and nonzero $\xi$ such that $\alpha\not\in \Pisot $ and $\lim_{n\to\infty}\|\xi\alpha^n\|=0$. 

Fractional parts of geometric progressions whose common ratios are Pisot numbers were studied in \cite{Ak-Ka:21} related to word combinatorics. However, it is generally difficult to generalize such results for geometric progressions whose common ratios are general algebraic numbers because the formulae for fractional parts have restrictions (see \cite{Ka:09}). In \cite{AKK}, a connection between word combinatorics and the fractional parts of general recurrence sequences was discovered. Therefore, it is natural to study the fractional parts of recurrence sequences rather than those of geometric progressions. 
For instance, we consider the recurrence sequences $x_n=\xi_1(10+3\sqrt{2})^n+\xi_2(10-3\sqrt{2})^{n}$ with $\xi_1,\xi_2\in \R$, where $(x_n)_{n\in \Z}$ satisfies the recurrence relation $x_{n+2}-20x_{n+1}+82x_n=0$ for every $n\in \Z$. In particular, we consider the maximal limit points $\limsup_{n\to\infty} \|x_n\|$ to investigate the fractional parts of $x_n$ ($n\geq 0$). It is more natural to consider the fractional parts of the recurrence sequence $(x_n)_{n \in \Z}$ instead of those of a geometric progression $(\xi (10+3\sqrt{2})^n)_{n\in \Z}$ with $\xi\in \R$. In fact, the set 
\[S:=\left\{\left. \limsup_{n\to\infty} \|\xi_1(10+3\sqrt{2})^n+\xi_2(10-3\sqrt{2})^{n}\| \ \right| \ (\xi_1,\xi_2)\in \R^2\right\} \]
 has the analogous property of the Markoff-Lagrange spectrum.  For instance, $S$ is a closed subset of $[0,1/2]$. However, it is not known whether the set 
\[\left\{\left. \limsup_{n\to\infty} \|\xi(10+3\sqrt{2})^n\| \ \right| \ \xi\in \R\right\}\]
is closed. Since the formulae for the fractional parts of recurrence sequences are flexible, we can also treat the fractional parts of the real parts of recurrence sequences. For instance, \cite{AKK} treats the fractional parts $\Re ((\xi_1n+\xi_0)\alpha^n))$ ($n\geq 0$), where $\xi_1,\xi_0$ are arbitrary complex numbers and $\alpha\approx-1.1495 + 2.3165  \sqrt{-1}$ is a zero of the polynomial $X^3 +2 X^2 +6 X - 2$. \par
In Section \ref{section2}, we give the main results that describe the Borel hierarchy related to the uniformity of the fractional parts of more general recurrence sequences. 
In Subsection \ref{subsec:3-2} we recall numeration systems which represent the fractional parts of recurrence sequences. The difficult point is that even when a recurrence sequence is uniformly distributed modulo one, it is difficult to analyze the average frequency of words in the digit expansions. Even if the initial values $\bg_1,\bg_2$ give recurrence sequences which are uniformly distributed modulo one, it is not known whether the average frequencies for the corresponding digit expansions coincides or if they even all exist. So we cannot directly apply the method in \cite{DSDB20} although the specification property holds on $\W_0$, the set of digits in our setting. 
In Section \ref{section4} we will analyze the fractional parts of recurrence sequences in terms of finite words. Lemma~\ref{lem:3-1} will play a crucial role for Wadge reduction.



\subsection{Borel hierarchy}

We now recall some basic notions from descriptive set theory which gauge the complexity of sets in Polish spaces.
In any topological space $X$, the collection of Borel sets $\sB(X)$ is the smallest
$\sigma$-algebra containing all open sets. Elements of $\sB(X)$ are stratified into levels,
introducing the Borel hierarchy on $\sB(X)$, by defining $\bS^0_1(X)$ to be the family of open sets,
and $\bP^0_1(X)
= \{ X\setminus A\colon A \in \bS^0_1(X)\}$ to be the family of closed sets. For a countable ordinal $\alpha<\omega_1$
we let $\bS^0_\alpha(X)$ be the collection of
countable unions $A=\bigcup_n A_n$ where each $A_n \in \bP^0_{\alpha_n}(X)$
for some ordinal $\alpha_n<\alpha$. We also let $\bP^0_\alpha(X)= \{ X\setminus A\colon A \in \bS^0_\alpha(X)\}$.
Alternatively, $A\in \bP^0_\alpha(X)$ if $A=\bigcap_n A_n$ where
$A_n\in \bS^0_{\alpha_n}(X)$ and $\alpha_n<\alpha$ for each $n$.
We also set $\bD^0_\alpha(X) = \bS^0_\alpha(X) \cap \bP^0_\alpha(X)$ for each countable ordinal $\alpha<\omega_1$. We will write  $\bS^0_\alpha, \bP^0_\alpha, \hbox{ and } \bD^0_\alpha$ when the Polish space $X$ is clear.

Note that $\bS^0_2$ is the collection of $F_\sigma$ sets,  $\bP^0_2$
is the collection of $G_\delta$ sets, and $\bD^0_1$ is the collection of clopen subsets.
For any topological space, $\sB(X)=\bigcup_{\alpha<\omega_1} \bS^0_\alpha=
\bigcup_{\alpha<\omega_1}\bP^0_\alpha$.  It is easy to see that all of the collections
$\bD^0_\alpha$, $\bS^0_\alpha$, $\bP^0_\alpha$ are {\em pointclasses}, that is, they are closed
under inverse images of continuous functions.
Another basic fact is that for any uncountable Polish space $X$, there is no collapse
in the levels of the Borel hierarchy, that is, all the
pointclasses $\bD^0_\alpha$, $\bS^0_\alpha$, $\bP^0_\alpha$, for any ordinal $\alpha <\omega_1$,
are distinct (for a proof, see \cite{Kechris}). Thus, these levels of the Borel hierarchy can be used
to calibrate the descriptive complexity of a set. We say a set
$A\subseteq X$ is $\bS^0_\alpha$ (resp.\ $\bP^0_\alpha$)-{\em hard}
if $A \notin \bP^0_\alpha$ (resp.\ $A\notin \bS^0_\alpha$). This says $A$ is
``no simpler'' than a $\bS^0_\alpha$ set. We say $A$ is $\bS^0_\alpha$-{\em complete}
if $A\in \bS^0_\alpha\setminus \bP^0_\alpha$, that is, $A \in \bS^0_\alpha$ and
$A$ is $\bS^0_\alpha$-hard. This says $A$ is exactly at the complexity level
$\bS^0_\alpha$. Likewise, $A$ is $\bP^0_\alpha$-complete if $A\in \bP^0_\alpha\setminus\bS^0_\alpha$.



Let us now discuss our proofs.
In order to determine the exact position of a set $A$ in the Borel hierarchy one must prove
an \emph{upper bound}, that is one must write a condition defining $A$
which shows that it appears at some level in the hierarchy,
and then show a \emph{lower bound}, that is, show
that $A$ does not belong to any lower level in the hierarchy.
To establish a lower bound we use a technique known as ``Wadge
reduction''. It is based on the observation that our hierarchy levels
are all pointclasses 
Thus, for example,  a Borel set $A$ is $\bS^0_\alpha$-hard
if there are a Polish space $Y$, a Borel set $C\subseteq Y$ which is known to be $\bS^0_\alpha$-hard,
and a continuous function $f\colon Y\to X$ such that $f^{-1}(A)=C$. 
The same holds for the $\bP^0_\alpha$
classes. 
Although the whole idea is plain and simple, the difficulty lies in the proper choice of the model space $Y$ and subset $C$, so that
it becomes possible to write down a definition of an appropriate continuous function.

\section{Main results}\label{section2}
In this paper, a {\it recurrence sequence} $(x_n)_{n\in\Z}$ is a sequence of complex numbers such that there exists a positive integer $D$ and integers $A_D,A_{D-1},\ldots,A_0$ with $A_D=1,A_0\ne 0$ such that 
\begin{align}\label{eqn:recurrence_basic}
\sum_{\ell=0}^D A_{\ell} x_{n+\ell}=0
\end{align}
for every integer $n$. The polynomial $f(X)=\sum_{\ell=0}^D A_{\ell} X^{\ell}$ is called the {\it companion polynomial}. 
For technical reasons, we assume the following throughout the paper: 
\begin{enumerate}[label=Assumption \arabic*., leftmargin=*]
\item $f(X)$ has the form $f(X)=P(X)^{k+1}$, where $k$ is a nonnegative integer and $P(X)=X^d+a_{d-1}X^{d-1}+\cdots+a_0\in \Z[X]$ ($d\ge 1, a_0\ne 0$) is a monic polynomial without multiple roots. 
\item For every complex number $z$ with modulus one, we have $P(z)\ne 0$. 
\end{enumerate}
Note that every irreducible factor 
$Q(X)\in \Z[X]\backslash\{0\}$ of $P(X)$ has at least one root 
with modulus greater than~1. 
Let $\alpha_1,\ldots,\alpha_d$ be the roots of $P(X)$, where $|\alpha_j|\ne 1,0$ for every $j=1,\ldots,d$. Then it is known that there exist polynomials $g_1(X),\ldots,g_d(X)\in \C[X]$ with $\deg g_j(X)\le k$ for $j=1,\ldots,p$ such that 
$x_n=\sum_{j=1}^d g_j(n)\alpha_j^n$
for every integer $n$ (see, for instance Chapter C of  \cite{ShTi1986} and Appendix F of \cite{BugeaudBook}). \par
We consider the uniformity of $e(\Re (x_n))$ ($n\ge 0$). Note that $(\Re (x_n))_{n\in\Z}$ also satisfies (\ref{eqn:recurrence_basic}). 
Instead of considering $(\Re (x_n))_{n\in\Z}$, we consider the case where $(x_n)_{n\in\Z}$ is a recurrence sequence of real numbers. If necessary, by changing the index, we may assume that $|\alpha_j|>1$ for $j=1,\ldots,p$ and $|\alpha_{\ell}|<1$ for $\ell=p+1,\ldots, d$, where $1\leq p\leq d$. Since $\sum_{j=q+1}^d g_j(n)\alpha_j^n$ does not affects the uniformity of $(e(x_n))_{n\ge 0}$, we may assume that $g_{q+1}(X)=\ldots=g_{d}(X)=0$. \par
Since $(x_n)_{n\in \Z}$ is a sequence of real numbers, we have by Lemma 1 (2) in \cite{AKK} the following: 
If $\alpha_j\in \R$ for $1\leq j\leq p$, then $g_j(X)\in \R[X]$. Moreover, if $1\leq j<\ell\leq p$ satisfy $\alpha_j=\ov{\alpha_{\ell}}$, then $g_j(X)=\ov{g_{\ell}(X)}$. We define $\Xi_k$ to be the set of tuples $(g_1(X),\ldots,g_p(X))\in \C[X]^p$ of polynomials of degree at most $k$ satisfying the above assumptions where for every $1\leq j<\ell\leq p$ with $\alpha_j=\ov{\alpha_{\ell}}\in \C\backslash \R$, we see that $g_{\ell}(X)$ is uniquely determined by $g_j(X)$. 
For instance, when $P(X)=X^3-2$, we have $\alpha_1=\sqrt[3]{2}, \alpha_2=\sqrt[3]{2}\zeta_3, \hbox{ and } \alpha_3=\ov{\alpha_2}$, where $\zeta_3$ is a primitive cubic root. Elements of $\Xi_k$ are vectors of the form $(g_1(X),g_2(X),\ov{g_2}(X))\in \R[X]\times \C[X]^2$, where $\deg g_1(X),\deg g_2(X)\le k$. 
Generally, we can identify $\Xi_k$ as $\R^{(k+1)r_1}\times \C^{(k+1)r_2}$ by considering the coefficients of $g_j(X)$, where $r_1=\mbox{Card}\{j\mid 1\leq j\leq p, \alpha_j\in \R\}$ and $r_2=(p-r_1)/2$. 
Thus, we introduce a topology on $\Xi_k$ by identifying $\Xi_k$ with the Euclidean space $\R^{(k+1)r_1}\times \C^{(k+1)r_2}$. 
Conversely, For every $\bg=(g_1(X),\ldots,g_p(X))\in \Xi_k$, we define the recurrence sequence $(x_n(\bg))_{n\in\Z}$ satisfying (\ref{eqn:recurrence_basic}) by 
\[
x_n(\bg)=\sum_{j=1}^p g_j(n) \alpha_j^{n}.
\]
Using Lemma 1 (2) in \cite{AKK} again, we see that $(x_n(\bg))_{n\in\Z}$ is a recurrence sequence of real numbers satisfying (\ref{eqn:recurrence_basic}). Let  
\[\bx^{+}(\bg):=(x_n(\bg))_{n\geq 0}.\] 
For convenience, we call $\bg$ the initial value for $(x_n(\bg))_{n\in\Z}$ throughout this paper. Set 
\[
\NN_{P,k}:=
\{
\bg\in\Xi_k\mid \bx^{+}(\bg)
\mbox{ is uniformly distributed modulo 1}
\}.
\]
We call any element $\bg\in \NN_{P,k}$ a { \it normal vector} for the recurrence sequence with companion polynomial $f(X)$.
For instance, when $P(X)$ is the minimal polynomial of $\alpha$, where $|\alpha|$ is a Pisot number, we have $p=1$ and we see that 
\begin{align*}
\Xi_k&=\left\{ \left. \sum_{j=0}^k\xi_j X^j
\ \right| \
\xi_0,\xi_1,\ldots,\xi_k\in \R
\right\},
\\
\NN_{P,k}&=
\left\{\sum_{j=0}^k\xi_j X^j\in \Xi_k
\ \left| \ 
\left(\sum_{j=0}^k \xi_j n^j \alpha^n\right)_{n\geq 0}
\mbox{is u.d. modulo 1}\right.\right\}.
\end{align*}
In particular, we have $\NN_{P,0}=\Na$. 
In the case where \[f(X)=P(X)=X^2-8X+14=(X-4-\sqrt{2})(X-4+\sqrt{2}) \ (k=0),\]
elements $(\xi_1,\xi_2)$ of $\NN_{P,0}\subseteq \Xi_0=\R^2$ are normal vectors for recurrence sequences with companion polynomial $f(X)$. 
We may now state our main theorem.
\begin{theorem}\label{theo:1-1}
The set $\NN_{P,k} \subseteq \Xi_k$ is $\bP_3^0(\Xi_k)$-complete assuming that $P$ and $k$ satisfy Assumption~1 and Assumption~2.
\end{theorem}
Theorem~\ref{theo:1-1} has several interesting consequences, including partially solving one of the original questions motivating this paper.
\begin{ex}\label{exam:1-1}
Let $\alpha$ be a real number such that $|\alpha|$ is Pisot. Then, for every 
$k\geq 0$, the set
\[
\left\{(\xi_0,\xi_1,\ldots,\xi_k)\in \R^{k+1}
\ \left| \ 
\left(\sum_{j=0}^k \xi_j n^j \alpha^n\right)_{n\geq 0}
\mbox{is u.d. modulo 1}\right.\right\}
\]
is $\bP_3^0(\R^{k+1})$-complete. In particular, the set of numbers normal in base $\alpha$ is $\bP_3^0(\R)$-complete.
\end{ex}

\begin{ex}\label{exam:1-2} The sets 
\begin{align*}
&\{(\xi,\eta)\in \R^2\mid 
(\xi(4+\sqrt{2})^n+\eta(4-\sqrt{2})^n)_{n\geq 0}
\mbox{ is u.d. modulo 1}
\}\\
&\{(\xi_0,\xi_1,\eta_0,\eta_1)\in \R^4\mid 
((\xi_0+\xi_1n)(4+\sqrt{2})^n+(\eta_0+\eta_1n)(4-\sqrt{2})^n)_{n\geq 0}
\mbox{ is u.d. modulo 1}
\}
\end{align*}
are $\bP_3^0(\R^2)$-complete and $\bP_3^0(\R^4)$-complete, respectively. 
\end{ex}
\begin{ex}\label{exam:1-3} The sets 
\begin{align*}
&\{\xi\in \C\mid 
(\xi(4+3\sqrt{-1})^n+
\ov{\xi}(4-3\sqrt{-1})^n
)_{n\geq 0}
\mbox{ is u.d. modulo 1}
\}
\\
&\{(\xi_0,\xi_1)\in \C^2\mid 
((\xi_0+\xi_1n)(4+3\sqrt{-1})^n+
(\ov{\xi_0}+\ov{\xi_1}n)(4-3\sqrt{-1})^n
)_{n\geq 0}
\mbox{ is u.d. modulo 1}
\}
\end{align*}
are $\bP_3^0(\C)$-complete and $\bP_3^0(\C^2)$-complete, respectively. 
\end{ex}

\begin{rem}
    It is obvious that $\NN_{P,k}$ is in $\bP_3^0(\Xi_k)$. In order to show that $\NN_{P,k}$ is $\bP_3^0(\Xi_k)$-complete, we will use Wadge reduction to show that it is $\bP_3^0(\Xi_k)$-hard.
\end{rem}
\section{Preliminaries}
In the rest of this paper, we fix $P(X)$ and $k\ge 0$. 
\subsection{Proof that $\NN_{P,k}$ is non-empty}
Weyl \cite{Weyl4} proved that for every real number $\alpha>1$, the sequence 
$(\xi \alpha^n)_{n\ge 0}$ is u.d. modulo one  
for almost all $\xi\in\R$. 
Many mathematicians have subsequently investigated 
multidimensional extensions of Weyl's theorem. 
For the proof of our results, it suffices to establish that $\NN_{P,k}$ in Theorem \ref{theo:1-1} is not empty. 
In this section we confine ourselves to proving this fact.
First we recall a simple version of 
metrical results of uniformly distribution by 
Koksma as follows: 
\begin{lem}[{Koksma \cite{Ko35}}]\label{lem:koksma}
    Let $(a_n)_{n\geq 1}$ be a sequence of real numbers satisfying the following: 
    there exists a positive constant $\delta$ such that, for any distinct positive integers $n$ and $m$, we have 
    \[|a_n-a_m|>\delta.\]
    Then, for almost every real number $\eta$, we have $(\eta a_n)_{n\geq 1}$ is uniformly distributed modulo one. 
\end{lem}
Let $\alpha$ be a zero of $P(X)$ with $|\alpha|>1$. When $\alpha$ is an imaginary number, let $K$ be a number field satisfying the following: 
\begin{enumerate}
    \item $\alpha,\sqrt{-1}\in K$.
    \item $K$ is a Galois extension of $\Q$.
\end{enumerate}
First we show the following: 
\begin{lem}\label{lem:argument}
    Suppose that $\alpha$ is imaginary. Then there exists an algebraic number $\xi_0$ such that for every $z\in K\backslash\{0\}$, we have $\xi_0 z\not\in \R$. 
\end{lem}
\begin{proof}
    For contradiction, we assume that there does not exist $\xi_0$ in Lemma \ref{lem:argument}. That is, for every algebraic number $\xi$, there exists $z\in K\backslash\{0\}$ such that $\xi z\in \R$. 
    In particular, for every odd prime number $p$, put $\zeta_p:=\exp(2\pi i/p)$. Then there exists $z_p\in K\backslash\{0\}$ such that $t_p:=\zeta_p^{-1} z_p$ is a positive real number. 
    Now put $z_p=:x_p+y_p\sqrt{-1}$ with $x_p,y_p\in \R$. Since $\sqrt{-1}, z_p\in K$ and $K$ is a Galois extension of $\Q$, we see that 
    \[x_p=\frac{z_p+\ov{z_p}}{2}\in K, \quad y_p=\frac{z_p-\ov{z_p}}{2\sqrt{-1}}\in K.\]
    Since \[t_p^2=|\zeta_p^{-1} z_p|^2=x_p^2+y_p^2\in K,\] we see 
    $[K(t_p):K]\leq 2$, where $[K(t_p):K]$ represents the degree of the field extension. Since $\zeta_p=z_p t_p^{-1}\in K(t_p)$, we get 
    \[[\Q(\zeta_p):\Q]\leq [K(t_p):\Q]\leq 2[K:\Q]\]
    for every odd prime number $p$, a contradiction. 
\end{proof}
In what follows, we will show that $\NN_{P,k}$ is non-empty by applying Lemma \ref{lem:koksma}. 
Put \begin{align*}a_n=
    \begin{cases}
        \alpha^n & (\alpha\in \R), \\
        \xi_0 \alpha^n+\ov{\xi_0} \ov{\alpha}^n & (\alpha\not\in \R),
    \end{cases}
\end{align*}
where $\xi_0$ is defined in Lemma \ref{lem:argument}. It suffices to show that $(a_n)_{n\geq 1}$ satisfies the assumption of Lemma \ref{lem:koksma}. In fact, Lemma \ref{lem:koksma} implies that $(\eta a_n)_{n\geq 1}$ is uniformly distributed modulo one for some real number $\eta$. 
When $\alpha\in \R$,  the assumption of Lemma~\ref{lem:koksma} is clear because $|\alpha|> 1$. In what follows, we assume that $\alpha\not\in \R$. \par
We will apply results on the growth of multi-recurrences in \cite{FH22}. 
Let 
\begin{align*}
f(n_1,n_2)&:=
\xi_0 \alpha^{n_1}\alpha^{n_2}+
\ov{\xi_0}\ov{\alpha}^{n_1}\ov{\alpha}^{n_2}
-\xi_0 \alpha^{n_1}\cdot 1^{n_2}-
\ov{\xi_0}\ov{\alpha}^{n_1}\cdot 1^{n_2}
\\
&=:
F_1(n_1,n_2)+F_2(n_1,n_2)+F_3(n_1,n_2)+F_4(n_1,n_2).
\end{align*}
By the assumption on $\xi_0$, we can check for any subset $I\subseteq \{2,3,4\}$ that 
\[
F_1(n_1,n_2)+\sum_{j\in I}F_j(n_1,n_2)\ne 0
\]
for any positive integers $n_1$ and $n_2$. 
For instance, suppose on the contrary that 
\[
F_1(n_1,n_2)+\sum_{j\in\{2,3,4\}}F_j(n_1,n_2)= 0.
\]
Then we have 
\[
\frac{\xi_0 (\alpha^{n_1+n_2}-\alpha^{n_1})}{\sqrt{-1}}\in \R,
\]
which contradicts the assumption on $\xi_0$. 
Thus, Theorem 1 in \cite{FH22} implies for every positive real number $\epsilon$ that 
there exists a positive constant $C(\epsilon)$ such that for any positive numbers $n_1$ and $n_2$ with $n_1+n_2\geq C(\epsilon)$, we have 
\[
|f(n_1,n_2)|\geq |\xi_0||\alpha|^{n_1+n_2}\exp(-\epsilon(n_1+n_2)).
\]
In particular, there exists a positive constant $C'$ such that, for all positive integers $n$ and $m$ with $\max\{n,m\}\geq C'$, we have 
\[
|a_n-a_m|\geq |\xi_0|\sqrt{|\alpha|}^{\max\{n,m\}}.
\]
Recalling that 
\[|a_n-a_m|\ne 0\]
for every distinct pair of positive integers $n$ and $m$ with $\max\{n,m\}< C'$, we see that there exists a positive constant $\delta$ satisfying the assumption of Lemma \ref{lem:koksma}. 
Therefore, we have proven that 
$\NN_{P,k}$ is non-empty.  
\subsection{Representations of fractional parts by using infinite words}\label{subsec:3-2}
In this subsection we introduce numeration systems connecting fractional parts of recurrence sequences $(x_n(\bg))_{n\in\Z}$ with initial values $\bg\in\Xi_k$ and the digit sequence $\bs$, where the set of digit sequences is a countable union of one-sided full shift spaces densely embedded in the two-sided full shift (see (\ref{eqn:fullshift})). First, we give formula (\ref{eqn:2-3}) for the fractional parts. This formula is denoted in terms of the series $\phi_m(\bs)$, where we use $\rho_m^{f}$ for $\phi_m(\bs)$ instead of $b^{-m}$ in the base-$b$ expansion with integer base $b$. Next, for every digit sequence $\bs$, we give an initial value $\theta(\bs)\in \Xi_k$ which realizes the digit sequence $\bs$ in the sense that the fractional parts of $x_m(\theta(\bs))$ represent $\phi_m(\bs)$ (see (\ref{eqn:2-4})). Conversely, for every initial value $\bg\in\Xi_k$, we construct a canonical digit sequence $\Psi(\bg)=(s_m(\bg))_{m\in \Z}$ which realizes the fractional parts of $(x_n(\bg))_{n\in\Z}$ in the sense that $\phi_m(\Psi(\bg))=e(x_m(\bg))$ (see (\ref{eqn:canonical})). 
In (\ref{eqn:2-14}) (see also Figure \ref{fig:commute}) we prove the commutative diagram related to the shift map $\sigma$ for the digit sequence $\bs$ and the map $\tau$ define by (\ref{eqn:2-11}). \par
Recall that $f(X)=\sum_{j=0}^DA_j X^j=P(X)^{k+1}$. In particular,  
$D=(k+1)d$, $A_D=1$, and $A_0\ne 0$. 
Let $\bg=(g_1(X),\ldots,g_p(X))\in \Xi_k$. 
Then $(x_n(\bg))_{n\in\Z}$ satisfies (\ref{eqn:recurrence_basic}), that is, for every integer $n$, we have  
\begin{align}\label{eqn:2-1}
\sum_{\ell=0}^D A_{\ell} x_{n+\ell}(\bg)
=0.
\end{align}
Put 
\[
\iota(\bg):=(x_n(\bg) \pmod{\Z})_{n\in \Z}\in (\R/\Z)^{\Z}
\]
and 
\[
\RR_k:=\{\iota(\bg)\mid \bg\in \Xi_k\}.
\]
\begin{prop}\label{prop:2-1}
The map $\iota: \Xi_k\to \RR_k$ is bijective. 
\end{prop}
\begin{proof}
It suffices to prove that $\iota$ is injective. 
Let $\bg^{(1)},\bg^{(2)}\in \Xi_k$ with $\iota(\bg^{(1)})=\iota(\bg^{(2)})$. 
Then, for every sufficiently large integer $m$, we have 
$|x_{-m}(\bg^{(j)})|<1/2$ for $j=1,2$, and so 
\[
x_{-m}(\bg^{(1)})=e\big(x_{-m}(\bg^{(1)})\big)
=
e\big(x_{-m}(\bg^{(2)})\big)=x_{-m}(\bg^{(2)}).
\]
Since 
\[
\sum_{\ell=0}^D A_{\ell} x_{n+\ell}(\bg^{( 1 )})
=
\sum_{\ell=0}^D A_{\ell} x_{n+\ell}(\bg^{( 2 )})
\]
for every integer $n$, we get by induction that, for every $n\in \Z$,  
\[
x_{n}(\bg^{(1)})=x_{n}(\bg^{(2)})
\]
and $\bg^{(1)}=\bg^{(2)}$ by Lemma 1(1) in \cite{AKK}. 
\end{proof}
Put $B:=\lfloor \sum_{\ell=0}^{D}2^{-1}\cdot |A_{\ell}|\rfloor$ and 
\[
\A:=[-B,B]\cap \Z.
\]
Moreover, let 
\begin{align}
\W_R&:=\left\{\left.\bs=(s_n)_{n\in\Z}\in \A^{\Z}
\ \right| \
s_{n}=0
\mbox{ for every }n<-R\right\}, 
\nonumber\\
\W&:=\bigcup_{R=0}^{\infty} \W_R.\label{eqn:fullshift}
\end{align}
Set 
\[
\rho_n^{f}:=\frac{-1}{2\pi\sqrt{-1}}\int_{|z|=1} \frac{z^{n-1}}{f(z)} dz
\]
for every integer $n$, 
where the contour integration is calculated counterclockwise. 
The next proposition follows from Proposition~1 in \cite{AKK}: 
\begin{prop}\label{prop:2-2}
\begin{enumerate}
    \item For every $n\geq 1$, we have 
\[
\rho_n^{f}=-\sum_{j=p+1}^{d} 
\mathrm{Res}\left(
\frac{z^{n-1}}{f(z)}, \alpha_j
\right),
\]
where $\mathrm{Res}(r(z),\lambda)$ is the residue of $r(z)$ at $z=\lambda$ 
and $\rho_n^{f}=0$ when $p=d$. \\
    \item For every $n\leq D-1$, we have 
\[
\rho_n^{f}=\sum_{j=1}^{p} 
\mathrm{Res}\left(
\frac{z^{n-1}}{f(z)}, \alpha_j
\right).
\]
\end{enumerate}
\end{prop}
Note that for every $j=1,\ldots,d$, we have that
\begin{align*}
\mathrm{Res}\left(
\frac{z^{n-1}}{f(z)}, \alpha_j
\right)
=
\left. \frac{1}{k!}\frac{d^k}{dz^k}(z^{n-1} G(z))\right|_{z=\alpha_j},
\end{align*}
where 
\[
G(z)=\frac{(z-\alpha_j)^{k+1}}{f(z)}
\]
with $G(\alpha_j)\ne 0,\infty$. For instance, if $k=0$, then 
\[
\mathrm{Res}\left(
\frac{z^{n-1}}{f(z)}, \alpha_j
\right)
=
\frac{\alpha_j^{n-1}}{f'(\alpha_j)}.
\]
For general $k\geq 0$, there exists $0<\delta<1$ such that 
\begin{align}\label{eqn:2-2}
\rho_n^{f}=o(\delta^{|n|})
\end{align}
as $|n|\to\infty$ because $|\alpha_j|>1$ for $j=1,\ldots,p$ and $|\alpha_j|<1$ for $j=p+1,\ldots,d$.  \par
For every $\bs=(s_n)_{n\in \Z}\in \W$, put 
\begin{align}\label{eqn:2-3}
\phi_m(\bs):=\sum_{n\in\Z} \rho_{m-n}^{f} s_n
\end{align}
for every integer $m$. Note that the right-hand side of (\ref{eqn:2-3}) 
absolutely converges by (\ref{eqn:2-2}). 
Moreover, set 
\[
\Phi(\bs):=(\phi_m(\bs)\pmod{\Z})_{m\in \Z}\in (\R/\Z)^{\Z}.
\]
By Theorem 1(3) in \cite{AKK} we have the following: 
\begin{fac}\label{fac:2-1}
The sequence $\Phi(\bs)$ is a member of $\RR_k$ for every $\bs\in \W$.
\end{fac}
The sequence $\Phi(\bs)$ realizes the fractional parts of certain recurrence sequences. 
In fact, let $\theta(\bs):=\iota^{-1}\circ 
\Phi(\bs)\in \Xi_k$ for $\bs\in \W$. 
Then we see by $\Phi(\bs)=\iota\circ \theta(\bs)$ that for every integer $m$, 
\begin{align}\label{eqn:2-4}
\phi_m(\bs)\equiv x_m\big(\theta(\bs)\big) \pmod{\Z}.
\end{align}
\begin{rem}\label{rem:aaaa}
    For every $\bs\in \W$, we see that $\theta(\bs)$ is the unique element in $\Xi_k$ satisfying (\ref{eqn:2-4}) for every integer $m$ since $\iota$ is injective. \par
\end{rem}
We recall the explicit representation for 
$\theta(\bs)$. 
Proposition 1(4) in \cite{AKK} implies that there exists 
$\bh=(h_1(X),\ldots,h_p(X))\in \Xi_k$ such that, for all $n\in \Z$,  
\[
\rho_n^f\equiv x_n(\bh) \pmod{\Z}.
\]
Set 
\[
h_j(X)=:\sum_{\ell=0}^{k} c_{j,\ell} X^{\ell}
\]
for every $j=1,\ldots, p$. Let $R\geq 0$ and 
$\bs=(s_n)_{n\in \Z}\in \W_R$. Put 
$\theta(\bs)=:(g_1(\bs;X),\ldots,g_p(\bs;X))$. By the proof of Theorem 1 (3) in \cite{AKK}, we have 
$g_j(\bs;X)=\sum_{i=0}^k \wi{{c}_{j,i}} X^i$, where 
\begin{align}\label{eqn:2-5}
\wi{{c}_{j,i}}=\sum_{\ell=i}^{k} \sum_{n=-R}^{\infty}
c_{j,\ell}{\ell\choose i}(-n)^{\ell-i}\alpha_j^{-n} s_n.
\end{align}
Now we introduce the discrete topology on $\A$. For every 
$R\geq 0$, we introduce the product topology for 
$\W_R\cong \A^{\Z\cap [-R,\infty)}$. Then the map 
\[
\theta |_{\W_R}:\W_R\mapsto \Xi_k
\]
is continuous because 
\[
\sum_{\ell=i}^{k} \sum_{n=-R}^{\infty}
|c_{j,\ell}|{\ell\choose i}n^{\ell-i}\alpha_j^{-n} B<\infty
\]
by $|\alpha_j|>1$ for $j=1,\ldots,p$.
Moreover, Theorem 1(2) in \cite{AKK} implies the following: 
\begin{fac}\label{fac:2-2}
The map $\theta: \W\to\Xi_k$ is surjective. 
\end{fac}
We give an explicit element $\Psi(\bg)$ in 
$(\emptyset\ne )\theta^{-1}(\bg)\subseteq \W$ for $\bg\in \Xi_k$. 
For every integer $m$, let 
\[
s_m(\bg):=\sum_{\ell=0}^D A_{\ell} u\big(x_{m+\ell}(\bg)\big).
\]
Then we see by (\ref{eqn:2-1}) that 
\begin{align*}
s_m(\bg)
=-\sum_{\ell=0}^D A_{\ell} e\big(x_{m+\ell}(\bg)\big)\in \A
\end{align*}
for every integer $m$ because $e(x_{i}(\bg))\in [-1/2,1/2)$ for all $i\in \Z$. Set  
\begin{align*}
\Psi(\bg):=(s_m(\bg))_{m\in \Z}\in \W
\end{align*}
because, for every sufficiently large $m$, we have $|x_{-m}(\bg)|<1/2$ and 
$s_{-m}(\bg)=0$. By Theorem 1(2) in \cite{AKK} we see 
for every integer $m$ that 
\begin{align}\label{eqn:canonical}
\phi_m\big(\Psi(\bg)\big)
=\sum_{n\in \Z} \rho_{m-n}^f s_n(\bg)
=e\big(x_m(\bg)\big)
\equiv x_m(\bg) \pmod{\Z}.
\end{align}
Applying Remark \ref{rem:aaaa} with $\bs=\Psi(\bg)$, we see that 
\[
\theta\circ\Psi(\bg)=\bg.
\]
For every $\bg=(g_1(X),\ldots,g_p(X))\in \Xi_k$, we set 
\begin{align}\label{eqn:2-11}
\tau(\bg):=(\alpha_1 g_1(X+1),\ldots, \alpha_p g_p(X+1)).
\end{align}
It is obvious that $\tau(\bg)\in \Xi_k$. 
Note that $\tau:\Xi_k\to\Xi_k$ is bijective because the inverse map is given by 
\[
\tau^{-1}(\bg)=(\alpha_1^{-1}g_1(X-1),\ldots,\alpha_p^{-1}g_p(X-1)). 
\]
Let $\sigma$ and $\wi{\sigma}$ be the shift operators on $\W$ and $\RR_k$, respectively. We have  
\begin{align}
    \sigma(\bs)&=(s_{n+1})_{n\in\Z}\in \W \mbox{ for } \bs=(s_n)_{n\in \Z}, \label{eqn:2-8}\\
    \wi{\sigma}\left((y_n \pmod{\Z})_{n\in\Z} \right)&=
    (y_{n+1} \pmod{\Z})_{n\in\Z} \mbox{ for } (y_n \pmod{\Z})_{n\in\Z}\in \RR_k. \nonumber
\end{align}
We now check that $\wi{\sigma}$ is well-defined, that is, for every $\bp \in\RR_k$, we have $\wi{\sigma}(\bp)\in \RR_k$. 
There exists $\bg=(g_1(X),\ldots,g_p(X))\in \Xi_k$ such that $\bp=(x_n(\bg)\pmod{\Z})_{n\in \Z}$. 
We see for every integer $m$ that 
\begin{align}\label{eqn:2-12}
x_{m+1}(\bg)=\sum_{j=1}^p(\alpha_j g_j(m+1))\alpha_j^m
=
x_m\big(\tau(\bg)\big).
\end{align}
In particular, we have $\wi{\sigma}(\bp)=(x_m(\tau(\bg))\pmod{\Z})_{m\in \Z}\in \RR_k$. Moreover, we get  
\begin{align}\label{eqn:2-13}
\wi{\sigma}\circ \iota=\iota \circ \tau.
\end{align}
We now show that
\begin{align}\label{eqn:2-10}
\Phi\circ \sigma=\wi{\sigma}\circ \Phi,
\end{align}
where $\mathrm{Im} \Phi \subseteq\RR_k$ by Fact \ref{fac:2-1}. 
In fact, let $\bs=(s_n)_{n\in \Z}\in \W$. It follows by (\ref{eqn:2-3}) that 
\begin{align}
\phi_m\big(\sigma(s)\big)
=
\sum_{n\in\Z} \rho_{m-n}^{f} s_{n+1}
=
\sum_{n\in\Z} \rho_{m+1-n}^{f} s_n
=
\phi_{m+1}(\bs)
\label{eqn:2-9}
\end{align}
for every integer $m$. 
Combining (\ref{eqn:2-13}) and (\ref{eqn:2-10}), we obtain 
by $\theta=\iota^{-1}\circ \Phi$ that $\iota^{-1}\circ \wi{\sigma}=\tau\circ\iota^{-1}$ and 
\begin{align}\label{eqn:2-14}
\theta \circ \sigma=\iota^{-1}\circ \Phi \circ \sigma=\iota^{-1}\circ \wi{\sigma}\circ \Phi
=\tau\circ\iota^{-1}\circ \Phi
=\tau\circ\theta.
\end{align}
\begin{figure}[ht]
\centering
\begin{tikzcd}
\W \arrow[r, "\sigma"] \arrow[d, "\theta"'] & \W \arrow[d, "\theta"] \\
\Xi_k \arrow[r, "\tau"'] & \Xi_k
\end{tikzcd}
\caption{$\sigma,\tau$ commute with $\theta$.}
\label{fig:commute}
\end{figure}
\begin{rem}
    The set $\W$ of digit sequences is a countable union of one-sided full shift spaces densely embedded in the two-sided full shift. $\W_0$ is naturally identified with a one-sided full shift space $\W':= \{(s_n)_{n\geq 0}\mid s_n\in \A \mbox{ for all }n\ge 0\}$. 
    However, when we consider special recurrence sequences such as geometric progressions, the digits are restricted. For example, let \[f(X)=X^2-20X+82=(x-10-3\sqrt{2})(x-10+3\sqrt2).\]
    We consider the case where $x_n= \xi_1 (10+3\sqrt{2})^n$ ($\xi_2=0$). In the same way as in the proof of  Proposition 5.1 in \cite{Ka:09} we can prove that $\sum_{n\in \Z}s_{n}((\xi_1,0))(10-3\sqrt{2})^{-n}=0$. Conversely, if $\bs=(s_n)_{n\in \Z}\in \W$ satisfies $\sum_{n\in\Z} s_n (10-3\sqrt{2})^{-n}=0,$ then there exists a real number $\xi_1$ such that $\phi_m(\bs)\equiv \xi_1 (10+3\sqrt{2})^m \pmod{\Z}$ for every $m\in \Z$ (see Proposition~5.2 in \cite{Ka:09}). Thus, it is interesting to investigate whether the set of such sequence of digits defined by \[
    \left\{(s_n)_{n\geq 0}\in \W'\ \left| \ s_n\in \A \mbox{ for every }n\ge 0, \sum_{n\geq 0}s_{n}(10-3\sqrt{2})^{-n}=0\right. \right\}
    \] satisfies the right feeble specification property.    The right feeble specification property is described in Definition~5 from~\cite{DSDB20}, where it is only given for subshifts. We note that this definition can easily be extended to subsets of the full shift to allow us to define it for $\W$. The right feeble specification property was key for proving the main theorem of \cite{DSDB20} and could likely be used to greatly extend the results of the current paper by combining elements from that proof.
    
\end{rem}
\section{Relation between uniform distribution modulo 1 and finite words}\label{section4}
Let 
\begin{align*}
\Gamma&:=\left\{\left.
\frac{a}{2^{\ell}} \ \right| \ \ell, a\in \Z, \ell\geq 1, -2^{\ell-1}\leq a\leq 2^{\ell-1}
\right\},
\\
\I&:=\left\{\left.
\left[\frac{a}{2^{\ell}}, \frac{b}{2^{\ell}}\right) 
\ \right| \
\ell,a,b\in \Z, \ell\geq 1, -2^{\ell-1}\leq a<b\leq 2^{\ell-1}
\right\}. 
\end{align*}
For every $I=[x,y)\in \I$, we denote its length by $|I|:=y-x$. Moreover, let 
\[
\mathcal{D}(I):=\min\{\ell\in \Z\mid \ell\geq 1, 2^{\ell}x,2^{\ell}y\in \Z\}.
\]
We list the elements of $\I$ by 
\[
\I=\{I_1,I_2,I_3,\ldots\}
\]
so that, for every $j,j'\geq 1$, if $\mathcal{D}(I_j)<\mathcal{D}(I_{j'})$, then $j<j'$. 
In particular, we have $\mathcal{D}(I_j)\leq j$ for every $j\geq 1$. \par
We now introduce the notation on finite words. 
Set 
\begin{align*}
\A^{+}:=\bigcup_{n\geq 1} \A^{\{0,1,\ldots,n-1\}}, \quad
\A^{*}:=\A^{+}\cup \{\emptyset\}, 
\end{align*}
where $\emptyset$ is the empty word. 
For every $v=v_0v_1\ldots v_{m-1}$, we denote its length by 
$|v|:=m$. For convenience, put $|\emptyset|:=0$. 
We define an equivalence relation $\sim$ on $\A^+$ as follows: 
For any $v,w\in \A^+$, we define $v\sim w$ if and only if 
there exists $\ell\geq 0$ such that $v=w0^{\ell}$ or $w=v0^{\ell}$. 
We denote the element of $\F:=\A^+/\sim$ including $v\in \A$ by 
$\ov{v}$. We define the map $\gamma:\F\to \W$ as follows: 
For any $\ov{v}=\ov{v_0v_1\ldots v_m}\in \F$, let 
$\gamma(\ov{v}):=(v_n)_{n\in\Z}$, where 
$v_n=0$ when $n\not\in \{0,1,\ldots,m-1\}$. 
Then $\gamma$ is well-defined and injective. 
For any $v\in \A^+$, we define 
\begin{align}\label{eqn:3-1}
\theta(v):=\theta\big(\gamma(\ov{v})\big).
\end{align}
For any $\bt=(t_n)_{n\in \Z}\in \W$ and $\ell\geq 1$, we put 
$\bt|_{\ell}:=t_0t_1\ldots t_{\ell-1}$. Moreover, for any integers $M,N$ with $M\leq N$, let 
$\bt|_{[M,N]}:=t_Mt_{M+1}\ldots t_N$. 
Similarly, for every 
$v=v_0v_1\ldots v_{m-1}$ and every integer $\ell$ with $1\leq \ell \leq m$, 
set $v|_{\ell}:=v_0v_1\ldots v_{\ell-1}$. \par
Let $\bg\in \Xi_k$ 
and $I\in \I$. 
For any positive integers $M, N$ with $M\leq N$, set 
\begin{align*}
\lambda(\bg,I;N)
&:=
\mbox{Card}\left\{
n\in \Z
\ \left| \ 
0\leq n\leq N-1, e\big(x_n(\bg)\big)\in I\right.\right\}
,\\
\lambda(\bg,I;M,N)
&:=
\mbox{Card}\left\{
n\in \Z
\ \left| \ 
M\leq n\leq N-1, e\big(x_n(\bg)\big)\in I\right.\right\}
,
\end{align*}
where $\lambda(\bg,I;0)=\lambda(\bg,I;N,N)=0$. 
Let $I=[a,b)\in \I$. 
We choose an interval $I_1=[a_1,b_1)\in\I$ with $I_1\subseteq I$. 
We take an interval $I_2$ of the torus $\R/\Z$ with $I \subseteq I_2$ as follows: 
If $-1/2<a<b<1/2$,  we take $I_2=[a_2,b_2)\in \I$ with 
$a_2<a<a_1<b_1<b<b_2$. 
If $a=-1/2, b<1/2$, then we use 
$I_2=[-1/2,b_2)\cup [-2^{-\ell}+1/2,1/2)$ with $b_2\in \Gamma, b<b_2<-2^{\ell}+1/2$ and 
$|I_2|:=2^{-\ell}+b_2+1/2$. 
If $a>-1/2,b=1/2$, then we use 
$I_2=[-1/2,-1/2+2^{-\ell})\cup [a_2,1/2)$ with $a_2\in \Gamma, -1/2+2^{-\ell}<a_2<a$ and 
$|I_2|:=2^{-\ell}+1/2-a_2$. 
If $a=-1/2,b=1/2$, then put $I_2:=[-1/2,1/2)$. 
Note that $I_1\in \I$, and $I_2\in \I$ or $I_2$ is a disjoint union of two elements of $\I$. 
\begin{lem}\label{lem:3-1}
There exists a positive integer $L=L(I,I_1,I_2)$, depending only on 
$I,I_1$ and $I_2$ satisfying the following: 
Let $N$ be positive integer and 
$\bs=(s_n)_{n\in \Z}, \bt=(t_n)_{n\in \Z}\in \W$ with 
$\bs|_N=\bt|_N$. Then we have 
\begin{align*}
\lambda(\theta(\bs),I;N)
&>
-L+\lambda(\theta(\bt),I_1;N)
, \\
\lambda(\theta(\bs),I;N)
&<
L+\lambda(\theta(\bt),I_2;N).
\end{align*}
\end{lem}
\begin{proof}
For the proof of Lemma \ref{lem:3-1}, we denote by $L'$ a positive 
integer depending only on $I,I_1$, and $I_2$, defined later. 
For all integers $i\leq j$, we denote 
\[
\bs_{[i,j]}:=s_is_{i+1}\ldots s_j, \quad 
\bt_{[i,j]}:=t_it_{i+1}\ldots t_j.
\]
In what follows, we may assume that $N\geq 5L'$. 
Recall that 
\[
\lambda(\theta(\bs),I;N)
=
\mbox{Card}\left\{0\leq n\leq N-1
\mid 
e\big(x_n(\theta(\bs))\big)\in I\right\}. 
\]
We take a sufficiently small $\epsilon=\epsilon(I,I_1,I_2)$, depending only on $I,I_1,I_2$, so that the following holds: 
Let $y_1,y_2$ be real numbers with $|y_1-y_2|<\epsilon$. 
If $e(y_1)\in I_1$ (resp. $e(y_1)\in I$), then $e(y_2)\in I$ (resp. $e(y_2)\in I_2$). 
By (\ref{eqn:2-2}) if $L'=L'(I,I_1,I_2)$ is sufficienlty large depending only on $I,I_1,I_2$, then 
\[
\sum_{m\in \Z\atop |m|\geq L'} B|\rho_m|<\frac{\epsilon}{2}. 
\]
Recall by (\ref{eqn:2-4}) and (\ref{eqn:2-3}) that 
\[
x_0(\theta(s))\equiv \phi_0(\bs)
=
\sum_{m\in \Z}\rho_{-m}^f s_m \pmod{\Z}.
\]
Using (\ref{eqn:2-12}) and (\ref{eqn:2-14}), 
we get for every nonnegative integer $n$ that 
\begin{align}\label{eqn:shift_on_x}
x_n(\theta(\bs))=x_0(\tau^n\circ\theta(s))
=x_0(\theta\circ \sigma^n(\bs)),
\end{align}
where $\tau^n, \sigma^n$ are the $n$-th iterations of $\tau, \sigma$, 
respectively. Let $L'\leq n\leq N-L'-1$, where $N-L'-1>L'$. Then we have 
\[
\sigma^n(\bs)|_{[-n,N-n-1]}=\sigma^n(\bt)|_{[-n,N-n-1]}.
\]
In particular, 
\begin{align}\label{eqn:3-2}
\sigma^n(\bs)|_{[-L',L']}=\sigma^n(\bt)|_{[-L',L']}.
\end{align}
Moreover, we see by (\ref{eqn:3-2}) that, for every $L'\leq n\leq N-L'-1$, 
\[
|x_0(\theta\circ \sigma^n(\bs)) - x_0(\theta\circ \sigma^n(\bt))|<\epsilon.
\]
Thus, we get the following: 
\begin{enumerate}
\item If $L'\leq n\leq N-L'-1$ and 
$e(x_0(\theta\circ \sigma^n(\bs)))\in I$, 
then 
$e(x_0(\theta\circ \sigma^n(\bt)))\in I_2$.
\item If $L'\leq n\leq N-L'-1$ and 
$e(x_0(\theta\circ \sigma^n(\bt)))\in I_1$, 
then 
$e(x_0(\theta\circ \sigma^n(\bs)))\in I$.
\end{enumerate}
Hence, we obtain 
\begin{align*}
\lambda(\theta(\bs),I;N)
&=
\mbox{Card}\{0\leq n\leq N-1
\mid 
e(x_0(\theta\circ \sigma^n(\bs)))\in I\}
\\
&\geq
\mbox{Card}\{L'\leq n\leq N-L'-1
\mid 
e(x_0(\theta\circ \sigma^n(\bs)))\in I\}
\\
&\geq
\mbox{Card}\{L'\leq n\leq N-L'-1
\mid 
e(x_0(\theta\circ \sigma^n(\bt)))\in I_1\}
\\
&\geq
-2L'+
\mbox{Card}\{0\leq n\leq N-1
\mid 
e(x_0(\theta\circ \sigma^n(\bt)))\in I_1\}
\\
&=
-2L'+\lambda(\theta(\bt),I_1;N).
\end{align*}
Similarly, 
\begin{align*}
\lambda(\theta(\bs),I;N)
&=
\mbox{Card}\{0\leq n\leq N-1
\mid 
e(x_0(\theta\circ \sigma^n(\bs)))\in I\}
\\
&\leq
2L'+
\mbox{Card}\{L'\leq n\leq N-L'-1
\mid 
e(x_0(\theta\circ \sigma^n(\bs)))\in I\}
\\
&\leq
2L'+
\mbox{Card}\{L'\leq n\leq N-L'-1
\mid 
e(x_0(\theta\circ \sigma^n(\bt)))\in I_2\}
\\
&\leq
2L'+
\mbox{Card}\{0\leq n\leq N-1
\mid 
e(x_0(\theta\circ \sigma^n(\bt)))\in I_2\}
\\
&=
2L'+\lambda(\theta(\bt),I_2;N).
\end{align*}
\end{proof}
\begin{rem}\label{rem:3-1}
Using (\ref{eqn:shift_on_x}), we obtain the following: 
For every pair of positive integers $M, N$ with $M\leq N$, we have 
\begin{align}
&\lambda(\theta(\bs),I;M,N)\nonumber\\
&=\mbox{Card}\{M\leq n\leq N-1\mid 
e(x_0(\theta\circ \sigma^n(\bs)))\in I\}
\nonumber\\
&=
\mbox{Card}\{0\leq m\leq N-M-1\mid 
e(x_0(\theta\circ \sigma^m(\sigma^M(\bs))))\in I\}
\nonumber\\
&=
\lambda(\theta(\sigma^M(\bs)),I;N-M).\label{eqn:3-3}
\end{align}
\end{rem}
We note for every $\bg\in \Xi_k$ that 
$\bx^+(\bg)=(x_n(\bg))_{n\geq 0}$ is uniformly distributed modulo one 
if and only if for every $I\in \I$, we have 
\[
\lim_{N\to\infty}\frac{1}{N}\lambda(\bg,I;N)=|I|.
\]
\begin{defi}\label{defi:3-1}
Let $\bt\in \W$. Then $\bt$ is {\it generic} if $\bx^+(\theta(\bt))$ is 
uniformly distributed modulo one. 
\end{defi}
\begin{defi}\label{defi:3-2}

\begin{enumerate}
\item Let $w\in \A^+$. 
Let $m$ be a positive integer and $\epsilon$ a positive real number. 
Then $w$ is $(m,\epsilon)$-good if, for every $j=1,\ldots,m$, we have 
\[
|I_j|-\epsilon<\frac{1}{|w|}\lambda(\theta(w),I_j;|w|)<|I_j|+\epsilon,
\]
where $\theta(w)$ is defined by (\ref{eqn:3-1}). \\
\item Let $(u_n)_{n\geq 1}$ be a sequence with  
$u_n\in A^+$ for every $n\geq 1$. 
Then $(u_n)_{n\geq 1}$ is generic if for every $I\in \I$, we have 
\[
\lim_{n\to\infty}\frac{1}{|u_n|}
\lambda(\theta(u_n),I;|u_n|)=|I|
\]
\end{enumerate}
\end{defi}
\begin{rem}\label{rem:bbbb}
The sequence $(u_n)_{n\geq 1}$ is generic if and only if for every $m\geq 1$ 
and $\epsilon>0$, 
there exists a positive integer $N(m,\epsilon)$ such that 
$u_n$ is $(m,\epsilon)$-good for every $n\geq N(m,\epsilon)$. 
Let $I\in \I$ and $I_1\subseteq I\subseteq I_2$ be as in Lemma \ref{lem:3-1}. 
Then $I_2$ may be a disjoint union $I_2=J_1\cup J_2$ with $J_1,J_2\in \I$. If $w$ is $(m,\epsilon/2)$-good, where 
$m$ is sufficiently large depending only on $J_1,J_2$, we have 
\[
|I_2|-\epsilon<\frac{1}{|w|}\lambda(\theta(w),I_2;|w|)<|I_2|+\epsilon.
\]
\end{rem}
\begin{defi}\label{defi:3-2}
Let $(u_n)_{n\geq 1}$ be a sequence with  
$u_n\in A^+$ for every $n\geq 1$. \\
\begin{enumerate}
\item The sequence $(u_n)_{n\geq 1}$ is dominating if  $(|u_1|+\cdots+|u_n|)/|u_{n+1}|$ $(n=1,2,\ldots)$ 
converges monotonically to 0.\\
\item The sequence $(u_n)_{n\geq 1}$ is asymptotically stable if the following holds: 
For every positive integer $m$ and positive real number $\epsilon$, 
there exists positive integer $N=N(m,\epsilon)$ such that 
if $n\geq N$, we can find $\ell'(n)$ such that 
\[
\ell'(n)<\min\{|u_{n+1}|,\epsilon|u_n|\} \mbox{ and }
\lim_{n\to\infty}\ell'(n)=\infty
\]
and that
\[
u_{n+1}|_{j} \mbox{ is $(m,\epsilon)$-good for every 
$\ell'(n)\leq j \leq |u_{n+1}|$}.
\]
\end{enumerate}
\end{defi}
\begin{lem}\label{lem:3-2}
Let $(u_n)_{n\geq 1}$ be a sequence with 
$u_n\in A^+$ for every $n\geq 1$. 
Assume that $(u_n)_{n\geq 1}$ is dominating and asymptotically stable. 
Let $\bt=(t_n)_{n\in\Z}\in \W_0\subseteq \W$ be defined as follows: 
\begin{align*}
\begin{cases}
t_n=0  \mbox{ if }n\leq -1, \\
(t_n)_{n\geq 0}=u_1u_2u_3\ldots.
\end{cases}
\end{align*}
Then $\bt$ is generic. 
\end{lem}
\begin{proof}
Note that $(u_n)_{n\geq 1}$ is generic. For any positive integer $n$, 
put $U_n:=u_1u_2\ldots u_n$. For every sufficiently large integer 
$\ell$, we have $\bt|_{\ell}$ is represented as $U_n v$, where 
$n\geq 2$ and 
$v=v(\ell)\in \A^*$ is a proper prefix of $u_{n+1}$. 
For the proof of Lemma \ref{lem:3-2}, it suffices to show for every 
$I\in \I$ and every positive real number $\epsilon$, the following holds: 
There exists a positive integer $n_1=n_1(I,\epsilon)$ depending only on 
$I$ and $\epsilon$ such that if $\ell\geq n_1$, then 
\begin{align}\label{eqn:3-4}
|I|-\epsilon\leq \frac1{\ell}\lambda(\theta(\bt),I;\ell)
\leq |I|+\epsilon.
\end{align}
We first show that the above claim holds when 
$\ell \geq n_2(I,\epsilon)$ with suitable constant $n_2(I,\epsilon)$ under the assumption that  
$\ell=|U_n|$ for some $n\geq 2$. Note that $v$ is the empty word. 
It suffices to consider the case where $\epsilon$ is sufficiently small 
depending on $I$. 
We take $I_1=I_1(I,\epsilon),  I_2=I_2(I,\epsilon)$ as in Lemma \ref{lem:3-1} with 
\begin{align}\label{eqn:3-5}
|I_2|<|I_1|+\frac{\epsilon}{4}.
\end{align}
We estimate lower bounds for $\lambda(\theta(\bt),I;\ell)$. 
Observe that 
\begin{align*}
\lambda(\theta(\bt),I;\ell)
&=
\lambda(\theta(\bt),I;|U_n|)
\\
&\geq 
\lambda(\theta(\bt),I;|U_{n-1}|,|U_n|)
\\
&=
\lambda(\theta(\sigma^{|U_{n-1}|}(\bt)),I;|u_n|),
\end{align*}
where we use (\ref{eqn:3-3}) and $|U_n|-|U_{n-1}|=|u_n|$ for the 
last equality. 
Note that $u_n$ is a prefix of $\sigma^{|U_{n-1}|}(\bt)$. 
Since $I_1$ and $I_2$ are chosen depending only on 
$I$ and $\epsilon$, we see by Lemma \ref{lem:3-1} with $N=|u_n|$ that 
there exists a positive integer $L=L(I,\varepsilon)$, depending only on 
$I$ and $\epsilon$, satisfying the following: 
\[
\lambda(\theta(\bt),I;\ell)
\geq 
-L+
\lambda(\theta(u_n),I_1;|u_n|).
\]
Note that $\lim_{n\to\infty}|u_n|/|U_n|=1$ because 
$(u_n)_{n\geq 1}$ is dominating. 
Since $(u_n)_{n\geq 1}$ is generic, if $\ell=|U_n|$ is 
sufficiently large depending only on $I$ and $\epsilon$, then we obtain by (\ref{eqn:3-5}) that 
\begin{align*}
\lambda(\theta(\bt),I_1;\ell)
&\geq 
-L+|u_n|\left(|I_1|-\frac{\epsilon}4\right)\\
&>
-L+|u_n|\left(|I|-\frac{\epsilon}2\right)\\
&>
-L+|U_n|\left(|I|-\frac{3\epsilon}4\right)\\
&>
|U_n| (|I|-\epsilon)=\ell (|I|-\epsilon).
\end{align*}
Similarly, we estimate upper bounds for $\lambda(\theta(\bt),I;\ell)$ as follows:  
Note that $\lim_{n\to\infty}|U_{n-1}|/|U_n|=0$ because 
$(u_n)_{n\geq 1}$ is dominating. Hence, if $\ell=|U_n|$ is sufficiently large depending only on $I$ and $\epsilon$, then 
\begin{align*}
\lambda(\theta(\bt),I;\ell)
&\leq 
|U_{n-1}|
+
\lambda(\theta(\sigma^{|U_{n-1}|}(\bt)),I;|u_n|)
\\
&\leq 
|U_{n-1}|
+L+
\lambda(\theta(u_n),I_2;|u_n|)\\
&\leq 
|U_{n-1}|
+L+
|u_n|\left(|I_2|+\frac{\epsilon}4\right)
\\
&\leq 
|U_{n-1}|
+L+
|U_n|\left(|I|+\frac{\epsilon}2\right)\\
&
<|U_n|(|I|+\epsilon)=\ell (|I|+\epsilon).
\end{align*}
In what follows, we show that (\ref{eqn:3-4}) holds for every 
$\ell\geq n_1(I,\epsilon)(\geq n_2(I,\epsilon))$. 
Note that $\ell=|U_n|+|v|$. 
We again take $I_1=I_1(I,\epsilon),I_2=I_2(I,\epsilon)$ as in Lemma \ref{lem:3-1} with 
(\ref{eqn:3-5}). Since $(u_n)_{n\geq 1}$ is asymptotically stable, 
we can take $\ell'=\ell'(n)$ for any $n\geq N(m,\epsilon/3)$ in Definition \ref{defi:3-2} (2), where we take suitable $m$ depending on $I_1$ and $I_2$. 
Note that 
\[
\ell'<\frac{\epsilon}{3}|u_n|<\frac{\epsilon}2 |U_n|.
\]
First, we consider the case where 
$
|v|<\ell'(<\epsilon|U_n|\cdot 2^{-1}).
$
Since we proved (\ref{eqn:3-4}) in the case where $\ell=|U_n|$, 
we see 
\begin{align*}
\lambda(\theta(\bt),I;\ell)
\geq \lambda(\theta(\bt),I;|U_n|)>|U_n|\left(|I|-\frac{\epsilon}2\right). 
\end{align*}
Observing that 
\begin{align*}
\frac{|U_n|}{\ell}=\frac{|U_n|}{|U_n|+|v|}>\frac{1}{1+\epsilon/2}
>1-\frac{\epsilon}2,
\end{align*}
we get by $|I|\leq 1$ that 
\[
\lambda(\theta(\bt),I;\ell)>
\ell\left(1-\frac{\epsilon}2\right)\left(|I|-\frac{\epsilon}2\right)
\\
\geq \ell(|I|-\epsilon). 
\]
Similarly, we obtain upper bounds as follows: 
\begin{align*}
\lambda(\theta(\bt),I;\ell)
&\leq 
|v|
+
\lambda(\theta(\bt),I;|U_n|)\\
&<
\frac{\epsilon}{2}|U_n|+|U_n|\left(|I|+\frac{\epsilon}{2}\right)
\\
&=|U_n|(|I|+\epsilon)\leq \ell (|I|+\epsilon).
\end{align*}
Next, we consider the case of $|v|\geq \ell'$. 
Observe by (\ref{eqn:3-3}) that 
\begin{align*}
&\lambda(\theta(\bt),I;\ell)\\
&=
\lambda(\theta(\bt),I;|U_n|)
+
\lambda(\theta(\bt),I;|U_n|,|U_nv|)
\\
&=
\lambda(\theta(\bt),I;|U_n|)
+
\lambda(\theta(\sigma^{|U_n|}(\bt)),I;|v|)\\
&=:y_1+y_2
\end{align*}
and that 
\[
y_1>(|I|-\epsilon)|U_n|.
\]
Recall that $\lim_{n\to\infty}\ell'(n)=\infty$ and $|v|\geq \ell'=\ell'(n)$. 
Thus, if $n_1(I,\epsilon)$ is sufficiently large and $\ell\ge n_1(I,\epsilon)$, then $3^{-1}\epsilon |v|>L$ and $v$ is $(m,\epsilon/3)$-good. Applying Lemma \ref{lem:3-1} with $N=|v|$, 
we obtain by $\sigma^{|U_n|}(\bt)|_{|v|}=|v|$ that 
\begin{align*}
y_2
&\geq -L+\lambda(\theta(v),I_1;|v|)
\\
&\geq
-L+|v|\left(|I_1|-\frac{\epsilon}{3}\right)
\\
&>
-L+|v|\left(|I|-\frac{2\epsilon}{3}\right)>|v|(|I|-\epsilon).
\end{align*}
Therefore, we deduce that 
\[
y_1+y_2>(|I|-\epsilon)(|U_n|+|v|)=\ell (|I|-\epsilon). 
\]
Similarly, we see 
\[
y_1<(|I|+\epsilon)|U_n|.
\]
Moreover, 
\begin{align*}
y_2
&\leq 
L+\lambda(\theta(v),I_2;|v|)\\
&\leq 
L+|v|\left(|I_2|+\frac{\epsilon}{3}\right)\\
&<
L+|v|\left(|I|+\frac{2\epsilon}{3}\right)<|v|(|I|+\epsilon).
\end{align*}
Finally, we we deduce that 
\[
y_1+y_2<(|I|+\epsilon)(|U_n|+|v|)=\ell (|I|+\epsilon). 
\]
\end{proof}
\begin{rem}\label{rem:cccc}
    In the proof of Lemma \ref{lem:3-2}, we compare $\lambda(\bt,I;N)$ with $\lambda(\bt|_N,I_1;N)$ and $\lambda(\bt|_N,I_2;N)$. In the same way we can show the following: Suppose that $\bt\in \W$ is generic. Put $u_n:=\bt|_n$ for $n=1,2,\ldots$. Then $(u_n)_{n\ge 1}$ is generic. 
\end{rem}

\section{Completion of the proof of Theorem \ref{theo:1-1}}
Here we will use {\it Wadge reduction} to show that $\NN_{P,k}$ is $\bP_3^0$-hard when $P$ and $k$ satisfy the hypotheses of Theorem~\ref{theo:1-1}. That is, let $X$ and $Y$ be Polish spaces and let $A \subseteq X$ and $B \subseteq Y$ along with a continuous function $f:Y \to X$ where $f^{-1}(A)=B$.  Then if $B$ is $\bS^0_\alpha$-complete (resp. $\bP^0_\alpha$-complete), then $A$ is $\bS^0_\alpha$-hard ($\bP^0_\alpha$-hard). The function $f$ reduces the question of membership in $A$ to membership in $B$.

Take $\bg\in \NN_{P,k}$, namely, $\bx^+(\bg)=(x_n(\bg))_{n\geq 0}$ is uniformly distributed modulo 1. 
Recall that $\Psi(\bg)=(s_m(\bg))_{m\in \Z}\in \W$. 
By (\ref{eqn:2-12}) we see for every $R\geq 0$ that $x_{n-R}(\bg)=x_{n}(\tau^{-R}(\bg))$ and $s_{n-R}(\bg)=s_n(\tau^{-R}(\bg))$, and so $\tau^{-R}(\bg)\in \NN_{P,k}$. Since $s_{-m}(\bg)=0$ for sufficiently large $m$, if necessary changing $\bg$ by $\tau^{-R}(\bg)\in \W$ with sufficiently large $R$, we may assume that $s_m(\bg)=0$ for every negative integer $m$. 
In particular, we have $\Psi(\bg)\in \W_0$. 
For simplicity, set 
\[
\bs=(s_m)_{m\in \Z}:=\Psi(\bg). 
\]
We denote the set of the sequences $(\beta(n))_{n\geq 1}$ of positive 
integers by $\CC=\N^{\N}$. We introduce the product topology on $\CC$ induced by the discrete topology on $\N$. For any $(\beta(n))_{n\geq 1}\in \CC$, we put 
\[
\beta'(n):=\min\{n,\beta(n)\}.
\]
For any $(\beta(n))_{n\geq 1}$, we choose sequences 
$(a_n)_{n\geq 0}$ and $(c_n)_{n\geq 0}$ of positive integers 
satisfying $a_0=c_0=1$ and the following assumptions: 
\begin{itemize}
\item[(a)] $a_n=\beta'(n) c_n$ for every $n\geq 1$.
\item[(b)] $c_n/n>2^{2n}$ for every $n\geq 1$.
\item[(c)] For any $n\geq 1$ every $m\geq c_n$, we have 
$\bs|_m$ is $(n,2^{-n-1})$-good.
\end{itemize}
Note by Remarks \ref{rem:bbbb} and \ref{rem:cccc} that assumption (c) holds when $c_n$ is sufficiently large.
Moreover, we also take a sequence $(b_n)_{n\geq 0}$ of 
nonnegative integers satisfying $b_0=0$ and choose such a sequence, sufficiently rapidly increasing, so that the following conditions are satisfied: for every $n\geq 1$, we have  
\begin{itemize}
\item[(A)] $b_n>2^{2n}$.
\item[(B)] $a_n b_n>2^{2n} a_{n+1}$.
\item[(C)] $a_n b_n>2^{2n}\sum_{j=1}^{n-1}(a_j+c_j) b_j$.
\item[(D)] The following sequence monotonically converges to 0: 
\[
\frac{1}{(a_{n+1}+c_{n+1})b_{n+1}}\sum_{j=1}^{n}(a_j+c_j) b_j \quad (n=1,2,\ldots).
\]
\end{itemize}
Moreover, let $B_0:=0$ and $B_n:=2(b_1+b_2+\cdots+b_n)$ for every 
positive integer $n$. \par
Using $(a_n)_{n\geq 0}, (b_n)_{n\geq 0}$ and $(c_n)_{n\geq 0}$, 
we define a sequence $(u_j)_{j\geq 1}$ of finite words on $\A$ 
as follows: First, 
\begin{align*}
u_j:=
\begin{cases}
\bs|_{a_1} & \mbox{ if }0<j\leq b_1,\\
0^{c_1} & \mbox{ if }b_1<j\leq 2b_1=B_1.
\end{cases}
\end{align*}
Moreover, for any $n\geq 1$, set 
\begin{align*}
u_j:=
\begin{cases}
\bs|_{a_{n+1}} & \mbox{ if }B_n<j\leq B_n+b_{n+1},\\
0^{c_{n+1}} & \mbox{ if }B_n+b_{n+1}<j\leq B_n+2b_{n+1}=B_{n+1}.
\end{cases}
\end{align*}
In this section, we do not treat $\F=\A^+/\sim$. 
Thus, $\ov{u}$ denotes an element in $\A^+$. 
We also define $(\ov{u_n}')_{n\geq 1}$ and $(\ov{u_n}'')_{n\geq 1}$ 
by 
\begin{align*}
\ov{u_n}'&=(\bs|_{a_n})^{b_n}=
\underbrace{\bs|_{a_n}\ldots \bs|_{a_n}}_{b_n\mbox{ times}}, 
\quad 
\ov{u_n}''=0^{b_n c_n}.
\end{align*}
We define the map $p:\CC\to \W_0$ by 
$p(\beta)=(t_n)_{n\in \Z}$, where  
\begin{align*}
\begin{cases}
t_n=0 & \mbox{ if }n\leq -1, \\
(t_n)_{n\geq 0}=u_1u_2u_3\ldots
=\ov{u_1}'\ov{u_1}''\ov{u_2}'\ov{u_2}''\ldots.
\end{cases}
\end{align*}
Let $N$ be a positive integer. Then $(a_n)_{n=0}^N$, $(b_n)_{n=0}^N$, $(c_n)_{n=0}^N$ can be constructed depending only on $(\beta(n))_{n=1}^{N+1}$. Thus, we can construct $(a_n)_{n\ge 0}$, $(b_n)_{n\ge 0}$, $(c_n)_{n\ge 0}$ so that $p$ is continuous. In what follows, we assume that $p$ is continuous. 
Finally, we define the map $\pi:\CC\to \Xi_k$ by 
$\pi=\theta\circ p$. Since $\theta|_{\W_0}$ is continuous, 
we see that $\pi$ is also continuous. 
Set 
\[
\CC_3:=\left\{(\beta(n))_{n\geq 1}\in \CC
\ \left| \
\lim_{n\to\infty}\beta(n)=\infty\right.\right\},
\]
which is a $\bP_3^0$-complete set. 
For the proof of Theorem \ref{theo:1-1}, 
it suffices to prove the following: 
\begin{lem}\label{lem:4-1}
Suppose that $\beta\in \CC_3$. Then the sequence 
$(\ov{u_n}' \ \ov{u_n}'')_{n\geq 1}$ is dominating and asymptotically 
stable. In particular, $\pi(\beta)\in \NN_{P,k}$ by Lemma \ref{lem:3-2}. 
\end{lem}
\begin{lem}\label{lem:4-2}
Suppose that $\beta\in \CC\backslash \CC_3$. 
Then $\pi(\beta)\not\in \NN_{P,k}$. 
\end{lem}
For the proof of Lemma \ref{lem:4-1}, we check the following: 
\begin{lem}\label{lem:4-3}
Suppose that $\beta\in \CC_3$. 
Let $m$ be a positive integer and $\epsilon$ a positive real number. 
Then there exists a positive integer $N=N(m,\epsilon)$, 
depending only on $m$ and $\epsilon$, satisfying the following: 
Let $n$ be an integer with $n\geq N(m,\epsilon)$. 
Put $\ell':=2^n a_n$. 
Then, for every integer $\ell$ with $\ell'\leq \ell\leq |\ov{u_n}'|$, we have 
$\ov{u_n}'|_{\ell}$ is $(m,\epsilon)$-good. 
\end{lem}
\begin{proof}
$\ov{u_n}'|_{\ell}$ is represented as $\wi{u_n}'\wi{u_n}''$, where 
$\wi{u_n}'=(\bs|_{a_n})^r$ with $r=\lfloor \ell/a_n\rfloor\geq 2^n$ and 
$\wi{u_n}''\in \A^*$ is a proper prefix of $\bs|_{a_n}$ with 
$|\wi{u_n}''|<a_n$. Let $I\in \I$ and $\epsilon>0$. 
For the proof of Lemma \ref{lem:4-3}, it suffices to verify that 
there exists a positive integer $n_1=n_1(I,\epsilon)$, 
depending only on $I$ and $\epsilon$, satisfying the following: 
If $n\geq n_1$, then, for every integer 
$\ell$ with $\ell'\leq \ell\leq |\ov{u_n}'|$, we have 
\begin{align}\label{eqn:4-1}
\ell(|I|-\epsilon)  < \lambda(\theta(\wi{u_n}'\wi{u_n}''),I;\ell)
< \ell(|I|+\epsilon).
\end{align}
It suffices to consider the case where $\epsilon$ is sufficiently 
small depending on $I$. 
We take $I_1,I_2$ with 
$
I_1\subseteq I\subseteq I_2
$
as in Lemma \ref{lem:3-1} so that 
\[
|I_2|-|I_1|<\frac{\epsilon}{4}. 
\]
We first estimate lower bounds for 
$\lambda(\theta(\wi{u_n}'\wi{u_n}''),I;\ell)$. 
We see 
\begin{align*}
\lambda(\theta(\wi{u_n}'\wi{u_n}''),I;\ell)
&\geq
\lambda(\theta(\wi{u_n}'\wi{u_n}''),I;ra_n)
\\
&=
\sum_{j=0}^{r-1}\lambda(\theta(\wi{u_n}'\wi{u_n}''),I;ja_n,(j+1)a_n).
\end{align*}
We identify \(\A^{*}\) with a subset of \(\W_0\) by mapping each finite word
\(t_0t_1\cdots t_{j-1}\in \A^{*}\) to the sequence \((t_n)_{n\in \Z}\in \W_0\) defined by
\(t_n=0\) for all \(n\not\in \{0,1,\ldots,j-1\}\).
Observe that for every $0\le j\le r-1$ that $\sigma^{ja_n}(\wi{u_n}'\wi{u_n}'')|_{a_n}=\bs|_{a_n}$.
Thus, combining Remark \ref{rem:3-1} and Lemma \ref{lem:3-1}, we get for every sufficiently large $n$ depending only on $I$ and $\epsilon$ that 
\begin{align*}
\lambda(\theta(\wi{u_n}'\wi{u_n}''),I;\ell)
&\geq 
\sum_{j=0}^{r-1}\lambda(\theta(\sigma^{ja_n}(\wi{u_n}'\wi{u_n}'')),I;a_n)
\\
&\geq 
\sum_{j=0}^{r-1}
\big(-L+\lambda(\theta(\bs),I_1;a_n)\big)
\\
&\geq 
\sum_{j=0}^{r-1}
\left(
-L+a_n\left(|I_1|-\frac{\epsilon}{4}\right)
\right)\\
&>
-rL+a_nr\left(|I|-\frac{\epsilon}{2}\right).
\end{align*}
Noting by $\ell\geq \ell'=2^n a_n$ that, for every sufficiently large $n$, 
\[
\frac{a_n r}{\ell}=\frac{a_n\lfloor \ell/a_n\rfloor}{\ell}
\geq 1-\frac{a_n}{\ell}\geq 1-\frac{\epsilon}{4},
\]
we obtain by $|I|\leq 1$ that 
\begin{align*}
\lambda(\theta(\wi{u_n}'\wi{u_n}''),I;\ell)
&> 
-\left\lfloor\frac{\ell}{a_n}\right\rfloor L
+
\ell\left(1-\frac{\epsilon}{4}\right)\left(|I|-\frac{\epsilon}{2}\right)\\
&\geq 
-\left\lfloor\frac{\ell}{a_n}\right\rfloor L
+
\ell\left(1-\frac{3\epsilon}{4}\right)\\
&>\ell(1-\epsilon),
\end{align*}
where we used the assumption that $\lim_{n\to\infty} a_n=\infty$. 
Similarly, we see 
\begin{align*}
\lambda(\theta(\wi{u_n}'\wi{u_n}''),I;\ell)
&\leq
|\wi{u_n}''|
+
\lambda(\theta(\wi{u_n}'\wi{u_n}''),I;ra_n)\\
&<
a_n
+
\sum_{j=0}^{r-1}\lambda(\theta(\wi{u_n}'\wi{u_n}''),I;ja_n,(j+1)a_n)\\
&\leq
a_n
+
\sum_{j=0}^{r-1}\lambda(\theta(\sigma^{ja_n}(\wi{u_n}'\wi{u_n}'')),I;a_n)
\\
&\leq
a_n
+
\sum_{j=0}^{r-1}
\big(L+\lambda(\theta(\bs),I_2;a_n)\big)\\
&\leq
a_n
+
\sum_{j=0}^{r-1}\left( L+\left(|I_2|+\frac{\epsilon}{4}\right)a_n\right)\\
&<
a_n+rL+\left(|I|+\frac{\epsilon}{2}\right)a_n r
\end{align*}
Using $\ell \geq 2^n a_n$ again, we deduce for every sufficiently 
large $n$ that 
\begin{align*}
\lambda(\theta(\wi{u_n}'\wi{u_n}''),I;\ell)
&<
a_n+\frac{\ell}{a_n}L +\left(|I|+\frac{\epsilon}{2}\right)\ell
<(|I|+\epsilon)\ell.
\end{align*}
Finally, we proved (\ref{eqn:4-1}). 
\end{proof}
\begin{proof}[Proof of Lemma \ref{lem:4-1}]
Put 
\begin{align*}
\ov{U_n}:=\ov{u_n}' \ \ov{u_n}''.
\end{align*}
Then the sequence $(\ov{U_n})_{n\geq 1}$ is dominating by 
assumption (D) and $|\ov{U_n}|=(a_n+c_n)b_n$. Let 
$I\in \I$ and $\epsilon>0$. 
For the proof of Lemma \ref{lem:4-1}, 
it suffices to show the existence of a positive integer 
$n_0=n_0(I,\epsilon)$, depending only on $I$ and $\epsilon$, satisfying 
the following: 
For any $n\geq n_0$, put $\ell':=2^n a_n$. Then for every integer $\ell$ 
with $\ell'\leq \ell\leq |U_n|$, we have 
\begin{align}\label{eqn:4-2}
|I|-\epsilon< \frac1{\ell}\lambda(\theta(\ov{U_n}|_{\ell}),I;\ell)
< |I|+\epsilon.
\end{align}
In fact, it is easily seen that 
\[
\lim_{n\to\infty}\frac{\ell'}{|\ov{U_n}|}=
\lim_{n\to\infty}\frac{2^n a_n}{(a_n+c_n)b_n}=0
\]
by Assumption (A). 
We may assume that $\epsilon$ is sufficiently small depending on 
$I$. In what follows, we denote $v:=\ov{U_n}|_{\ell}$ for simplicity. 
If $\ell\leq |\ov{u_n}'|=a_nb_n$, then (\ref{eqn:4-2}) follows from 
Lemma \ref{lem:4-3}. In what follows, we suppose that 
\[
a_nb_n<\ell\leq |\ov{U_n}|=b_n(a_n+c_n).
\]
Let $I_1,I_2$ be as in Lemma \ref{lem:3-1} with 
$I_1\subseteq I\subseteq I_2$ and 
\[
|I_2|<|I_1|+\frac{\epsilon}4.
\]
Using Lemma \ref{lem:3-1}, we estimate lower bounds for 
$\lambda(\theta(v),I;\ell)$ as follows: 
\begin{align*}
\lambda(\theta(v),I;\ell)
\geq 
\lambda(\theta(v),I;|\ov{u_n}'|)
\geq 
\lambda(\theta(\ov{u_n}'),I_1;|\ov{u_n}'|)-L
\end{align*}
Using (\ref{eqn:4-1}) in the case of $\ell=|\ov{u_n}'|$, we see 
\begin{align*}
\lambda(\theta(v),I;\ell)
&\geq 
|\ov{u_n}'|\left(|I_1|-\frac{\epsilon}{4}\right)-L
\\
&= 
a_nb_n\left(|I_1|-\frac{\epsilon}{4}\right)-L
\\
&>
a_nb_n\left(|I|-\frac{\epsilon}{2}\right)-L
\geq 
a_nb_n\left(|I|-\frac{3\epsilon}{4}\right).
\end{align*}
Recall that 
$\lim_{n\to\infty}\beta'(n)=\infty$ by $(\beta(n))_{n\geq 1}\in \CC_3$. 
Noting that 
\[
\frac{a_nb_n}{\ell}\geq \frac{a_nb_n}{b_n(a_n+c_n)}
=\frac{1}{1+(\beta'(n))^{-1}},
\]
we obtain for every sufficiently large $n$ that 
\[
\lambda(\theta(v),I;\ell)>\ell(|I|-\epsilon).
\]
Similarly, applying Lemma \ref{lem:3-1}, we obtain for every sufficiently 
large $n$ that 
\begin{align*}
\lambda(\theta(v),I;\ell)
&\leq 
\lambda(\theta(v),I;|\ov{u_n}'|)+b_nc_n\\
&\leq 
\lambda(\theta(\ov{u_n}'),I_2;|\ov{u_n}'|)+L+b_nc_n
\\
&\leq 
a_nb_n\left(|I_2|+\frac{\epsilon}{4}\right)+L+b_nc_n
\\
&<
a_nb_n\left(|I|+\frac{\epsilon}{2}\right)+L+b_nc_n.
\end{align*}
Since 
\[
\lim_{n\to\infty}\frac{c_n}{a_n}=\lim_{n\to\infty}\frac{1}{\beta'(n)}=0,
\]
we deduce for every sufficiently large $n$ that 
\begin{align*}
\lambda(\theta(v),I;\ell)
<
a_nb_n(|I|+\epsilon)< \ell(|I|+\epsilon).
\end{align*}
\end{proof}
\begin{proof}[Proof of Lemma \ref{lem:4-2}]
Since $(\beta(n))_{n\geq 1}\in \CC\backslash \CC_3$, there exists 
a positive integer $M$ such that there exists infinitely many positive 
integers $n$ with $\beta'(n)=M$. In what follows, we assume that 
$\beta'(n)=M$. Put 
\begin{align*}
V_n:=\ov{u_1}' \ \ov{u_1}'' \ 
\ov{u_2}' \ \ov{u_2}'' \ \ldots \ov{u_n}' \ \ov{u_n}''.
\end{align*}
We take a sufficiently large integer $\ell$ depending only on $M$ so that 
\begin{align}\label{eqn:qqqq}
    \frac{1}{4(M+1)}>2^{2-\ell}.
\end{align}
Set 
\[
I:=\left[-\frac{1}{2^{\ell}} , \frac{1}{2^{\ell}} \right]
, \quad
I_1:=\left[-\frac{1}{2^{1+\ell}} , \frac{1}{2^{1+\ell}} \right].
\]
We estimate lower bounds for $\lambda(\pi(\beta),I;|V_n|)$. 
Putting $V_n=:V_n' \ov{u_n}''$, 
we see 
\begin{align*}
\lambda(\pi(\beta),I;|V_n|)
&\geq 
\lambda(\theta(p(\beta)),I;|V_n'|,|V_n|)\\
&=\lambda(\theta(\sigma^{|V_n'|}\circ p(\beta)),I;|\ov{u_n}''|)\\
&=\lambda(\theta(\sigma^{|V_n'|}\circ p(\beta)),I;b_nc_n)
\end{align*}
by Remark \ref{rem:3-1}. We denote by $\mathbf{0}$ 
the constant sequence of zeros in $\W$. 
Since 
\[
\sigma^{|V_n'|}\circ p(\beta)|_{|\ov{u_n}''|}=\ov{u_n}''=0^{b_nc_n},
\]
we get by Lemma \ref{lem:3-1} with $N=b_nc_n$ that 
\begin{align*}
\lambda(\pi(\beta),I;|V_n|)
&\geq 
-L+\lambda(\theta(\mathbf{0}),I_1;b_nc_n)\\
&=-L+b_nc_n=-L+\frac{a_nb_n}{M}
\geq 
\frac{a_nb_n}{2M}
\end{align*}
for every sufficiently large $n$, where $L$ is a positive constant depending only on $M$ because $I$ and $I_1$ can be constructed depending only on $M$. 
On the other hand, Assumption (C) implies that 
if $n$ is sufficiently large, then 
\begin{align*}
|V_n|= \sum_{j=1}^{n}(a_j+c_j) b_j\leq 2(a_n+c_n)b_n=
2\left(1+\frac1{M}\right)a_nb_n.
\end{align*}
Therefore, we obtain for infinitely many $n\geq 1$ that 
\begin{align*}
\lambda(\pi(\beta),I;|V_n|)
&\geq 
\frac{1}{4(M+1)}|V_n|> 2^{2-\ell}|V_n|=2|I|\cdot |V_n| 
\end{align*}
by (\ref{eqn:qqqq}).
Hence, we deduce that $\pi(\beta)\not\in \NN_{P,k}$. 
\end{proof}

\section{Open problems}

\begin{prob}
For all real numbers $\alpha$ with $|\alpha|>1$, the set $\Na$ is $\bP_3^0$-complete.
\end{prob}

\begin{prob}
Recall that $\Pisot$ is the set of Pisot numbers. Generalize the results of \cite{BecherHeiberSlamanAbsNormal,BecherSlamanNormal} and 
show that $\bigcup_{\alpha\in \R, |\alpha| \in \Pisot} \Na$ is $\bS_4^0$-complete and that $\bigcap_{\alpha\in \R, |\alpha| \in \Pisot} \Na$ is $\bP_3^0$-complete. 
\end{prob}

\begin{prob}
    Let $P(X)\in \mathbb{Z}[X]$ be a monic irreducible polynomial such that all zeros $\alpha_1,\ldots,\alpha_p$ of $P(X)$ with absolute values greater than one are real numbers. Assume that $p>1$. Let $a$ be a positive integer with $a<p$. Show that the set 
    $$\{(\xi_1,\ldots,\xi_a)\in \mathbb{R}^a\mid \xi_1 \alpha_1^n+\cdots+\xi_a\alpha_a^n \mbox{ is u.d. modulo one.}\}$$
    is $\bP_3^0(\R^a)$-complete.
    Note: If $\alpha_1=4+\sqrt{2},\alpha_2=4-\sqrt{2},$ then this problem reduces to the problem of showing that $\NS{4+\sqrt{2}}$ is $\bP_3^0(\R)$-complete.
\end{prob}

A set $D\subseteq X$ is in the class $D_2(\bP^0_3)$ if $D=A \backslash B$ where $A, B\in \bP^0_3$.
A set $D$ is {\it $D_2(\bP^0_3)$-hard} if $X \backslash D \notin D_2(\bP^0_3)$, and
$D$ is {\it $D_2(\bP^0_3)$-complete} if it is in $D_2(\bP^0_3)$ and is $D_2(\bP^0_3)$-hard.
As with the classes $\bS^0_\alpha$, $\bP^0_\alpha$, the class $D_2(\bP^0_3)$
has a universal set and so is non-selfdual, that is, it is not closed under
complements.

In \cite{JacksonManceVandehey} it was shown that the set of numbers normal with respect to the continued fraction expansion that aren't normal in base $b$ is $D_2(\bP^0_3)$-complete and for any pair of relatively prime integers $p,q \geq 2$, the set of numbers normal in base $p$, but not normal in base $q$ is $D_2(\bP^0_3)$-complete. Other similar results involving Cantor series expansions can be found in \cite{AireyJacksonManceComplexityCantorSeries}.

The set $A \backslash B$ being $D_2(\bP^0_3)$-complete can be thought of as a notion of independence between the sets $A$ and $B$. In particular, if $A \cap C=B$, then the set $C$ must not be a $\bS_3^0$ set. This eliminates many possible theorems and can be used to show existence of several types of numbers. For example, the set of badly approximable numbers is a $\bS_2^0$ set, so we immediately know that there exist numbers that are normal in base $3$ and badly approximable that are not normal in base $2$. 

A real number is called {\it essentially non-normal in base $b$} if none of the digits $0,1,\ldots,b-1$ have a frequency in its $b$-ary expansion. It is easy to see that the set of real numbers that are essentially non-normal in base $b$ is a $\bS_3^0$ set. Thus, we immediately know by \cite{JacksonManceVandehey} that there exist real numbers that are continued fraction normal and essentially non-normal in base $3$ that are not normal in base $2$.

\begin{prob} 
Show that when $\Na \backslash \Nb \neq \emptyset$, we have that $\Na \backslash \Nb$ is $D_2(\bP_3^0)$-complete. For instance, it is known that $\NN (\sqrt{10})\subsetneq \NN(10)$. For every real number $\alpha>1$ and positive integers $r,s$, Brown, Moran, Pollington \cite{BrMoPo} gave a necessary and sufficient condition so that $\NN(\alpha^s)\subseteq \NN(\alpha^r)$. Moreover, Moran and Pollington \cite{MoPo} showed that if $\Na \subseteq \Nb$, then $(\log \alpha)/(\log \beta)$ is rational. 
\end{prob}

\begin{prob}
    Let $a,b$ be integers greater than one. Assume that $a$ and $b$ are multiplicatively independent. Put 
    $P_a(X)=X-a$ and $P_b(X)=X-b$. 
    Then for any nonnegative integer $k$, 
    $\NN_{P_a,k}\backslash\NN_{P_b,k}$
    is $D_2(\bP_3^0)$-complete.
\end{prob}

\bibliographystyle{amsplain}


\input{output.bbl}

\end{document}

%% file: output.bbl
\providecommand{\bysame}{\leavevmode\hbox to3em{\hrulefill}\thinspace}
\providecommand{\MR}{\relax\ifhmode\unskip\space\fi MR }
\providecommand{\MRhref}[2]{%
  \href{http://www.ams.org/mathscinet-getitem?mr=#1}{#2}
}
\providecommand{\href}[2]{#2}